\documentclass[12pt, oneside]{amsart}

\usepackage[utf8]{inputenc}
\usepackage[T1]{fontenc}

\usepackage{amssymb}
\usepackage{amsthm}
\usepackage{amsmath}
\usepackage[only,Yup]{stmaryrd} 
\usepackage{mathrsfs}    
\usepackage[new]{old-arrows}   
\usepackage{comment}
\usepackage{soul}
\usepackage{cancel}
\usepackage{bm}
\usepackage{biblatex} 
\usepackage[a4paper,left=2cm,top=3cm,right=2cm,bottom=2cm]{geometry}

\usepackage[all,cmtip]{xy}

\usepackage{enumerate}

\usepackage{indentfirst}

\usepackage{hyperref}
\hypersetup{colorlinks=true,allcolors=black}  

\usepackage{pstricks}
\usepackage{graphicx}

\usepackage{multicol}

\newtheorem{teo}{Theorem}[section]

\newtheorem{exe}{Example}[section]
\newtheorem{rmk}{Remark}[section]

\newtheorem{lemma}{Lemma}[section]
\newcommand{\qcd}{\begin{flushright} $\square$ \end{flushright}}

\newcommand{\f}{\mathcal{F}}

\mathchardef\ordinarycolon\mathcode`\:
\mathcode`\:=\string"8000
\begingroup \catcode`\:=\active
  \gdef:{\mathrel{\mathop\ordinarycolon}}
\endgroup

\begin{document}

\title{Transverse stable causality in Lorentzian foliations}

\author{H. A. Puel Martins}
\address{Departamento de Matemática, Universidade Federal de Santa Catarina, R. Eng. Agr. Andrei Cristian Ferreira, 88040-900, Florianópolis - SC, Brazil}
\email{henrique.martins@.ufsc.br}

\author{I. P. Costa e Silva}
\address{Departamento de Matemática, Universidade Federal de Santa Catarina, R. Eng. Agr. Andrei Cristian Ferreira, 88040-900, Florianópolis - SC, Brazil}
\email{}

\author{V. L. Espinoza}
\address{Departamento de Matemática, Universidade Federal de Santa Catarina, R. Eng. Agr. Andrei Cristian Ferreira, 88040-900, Florianópolis - SC, Brazil}
\email{}

\subjclass[2020]{53C12, 53C50}

\newenvironment{proofoutline}{\proof[Proof outline]}{\endproof}
\newenvironment{proofcomment}{\proof[Comment on the proof]}{\endproof}
\theoremstyle{definition}
\newtheorem{example}{Example}[section]
\newtheorem{definition}[example]{Definition}
\newtheorem{remark}[example]{Remark}
\theoremstyle{plain}
\newtheorem{proposition}[example]{Proposition}
\newtheorem{theorem}[example]{Theorem}
\newtheorem{corollary}[example]{Corollary}
\newtheorem{claim}[example]{Claim}
\newtheorem{conjecture}[example]{Conjecture}
\newtheorem{thmx}{Theorem}
\renewcommand{\thethmx}{\Alph{thmx}} 
\newtheorem{corx}[thmx]{Corollary}
\renewcommand{\thecorx}{\Alph{corx}} 

\newcommand{\dif}[0]{\mathrm{d}}
\newcommand{\od}[2]{\frac{\dif #1}{\dif #2}}
\newcommand{\pd}[2]{\frac{\partial #1}{\partial #2}}
\newcommand{\dcov}[2]{\frac{\nabla #1}{\dif #2}}
\newcommand{\proin}[2]{\left\langle #1, #2 \right\rangle}
\newcommand{\g}[0]{\mathcal{G}}
\newcommand{\metric}{\ensuremath{\mathrm{g}}}
\begin{abstract}
We introduce and investigate a notion analogous to stable causality for Lorentzian foliations, considerably extending the framework of transverse causality initiated in \cite{us}. It is well-known that in spacetime geometry stable causality is characterized by several equivalent conditions, such as the stability of non-existence of causal loops under metric perturbations, the existence of smooth temporal functions, and relation-theoretic formulations via the so-called Seifert and $K^+$ relations. We define natural transversal analogues of these distinct formulations and establish partial logical implications between them in the foliation setting. Whether and to what extent the full equivalences can be established remains an open question for general foliations, partly due to the eventual non-Hausdorff nature and other topological complexities of arbitrary leaf spaces. Furthermore, we demonstrate that for the important class of \textit{simple} foliations, which are defined by certain submersions, the equivalences of almost all the transverse analogues of stable causality are indeed recovered. 
\end{abstract}

\maketitle
\setcounter{tocdepth}{1}
\tableofcontents

\section{Introduction}

Given an $n$-manifold $M=M^n$ and an integer $0\leq q\leq n$, consider a \textit{(regular) codimension $q$ foliation } of $M$, by which we mean a partition $\f$ of $M$ by codimension $q$ immersed connected submanifolds - the \textit{leaves} of $\f$ - satisfying a suitable local regularity condition basically meaning that $\f$ is ``locally given by the fibers of a submersion'' (see more details below). A vector $v \in T_xM$ is \textit{vertical} (with respect to $\f$ if it is tangent to the leaf $L_x$ through $x \in M$, and the set of such vertical vectors mesh together as an $(n-q)$-dimensional distribution $T\f\subset TM$ called the \textit{vertical distribution} of $\f$. Any smooth section of $T\f$
is a \textit{vertical vector field}, and the collection of vertical vector fields on $M$ form a Lie subalgebra of $\mathfrak{X}(M)$ by Frobenius' theorem.  

The vertical distribution naturally defines a rank $q$ vector bundle on $M$ whose fiber at $x\in M$ is $T_xM/T\f_x$, the \textit{normal bundle} $\nu\f$ of the foliation. Geometric structures on $\nu \f$ which are invariant by the local flow of vertical vector fields are generally called \textit{transverse (or $\f$-basic)} structures. Many types of transverse structures have been studied in the literature, such as transverse orientability \cite{PMIHES_1981__54__5_0,Zung_2024}, $\f$-basic cohomology \cite{GoertschesToben+2018+1+40, Habib_2024,  royo2005top}, transverse symplectic structure \cite{10.1093/qmath/haaa071, 10.1093/qmath/hax051}, and  transverse semi-Riemannian geometry \cite{dolgonosova2018pseudo, caramello2025hadamardtheoremtransverselyaffine, caramello2024transversegeometrylorentzianfoliations}. 

The reason why transverse structures are an important part of the foliation-theoretic lore is twofold. First, such structures immediately generalize the corresponding structures on $M$, as they reduce to the latter in the trivial case of a codimension $q=n$ foliation by points. Second, the invariance of these structures by local flows of vertical vector fields means that they can be quite naturally interpreted as inducing a sort of ``low-regularity'' geometric structure on the \textit{leaf space} of the foliation  (i.e., the quotient space of the equivalence relation given by identifying points on the same leaf). Accordingly, the foliation-theoretic philosophy adopted here is that \textit{transverse geometry equals geometry on the leaf space}. 

Leaf spaces provide a broad class of low-regularity topological spaces, which includes not only smooth manifolds themselves, but also another fairly well-known ``irregular'' geometric structure which generalizes them, the so-called \emph{orbifolds}. 
These low-regularity objects appear in various parts of mathematics as orbit spaces of locally free Lie group actions; orbifolds arise in the theory of 3-manifolds as quotients of Seifert fiber spaces, and they also naturally emerge in some areas of mathematical physics, for example in phase spaces that are quotients of infinite-dimensional spaces by symmetry groups, such as Yang-Mills connections quotients by gauge groups in field theories, or moduli spaces of metrics in general relativity \cite{emmrich}. Although orbifolds form a relatively small subclass of leaf spaces, they have an especially rich geometric theory that includes generalizations of (semi-)Riemannian metric notions in a number of ways, such as geodesics and curvature, with many results analogous to those of usual geometry on manifolds, e.g., a version of the Bonnet-Myers finite diameter theorem for orbifolds. (See \cite{caramello3} for an excellent review of the key geometrical and topological properties of orbifolds, including extensive classic and recent references.)  

Another particularly well-established type of transverse structure appear in the so-called \textit{(regular) Riemanniann foliations}, consisting of a regular foliation $\f$ on $M$ together with a Riemannian fiber metric $\sigma$ on the normal bundle $\nu \f$ invariant by the local follows of vertical fields, i.e., $\mathcal{L}_V\sigma =0$ for any vertical vector field $V$, where $\mathcal{L}_V$ denotes the Lie derivative. Riemannian foliations have a well-developed, beautiful structural theory due P. Molino \cite{molino, molino2}. However, the also very natural case of \textit{Lorentzian foliations} (for which $\sigma$ is taken to be a Lorentzian metric instead of Riemannian) has hitherto remained largely unexplored (but see \cite{russas}). As pointed out above, the study of Lorentzian foliations can be viewed as a low-regularity generalization of Lorentzian geometry.  

Now, the broader study of Lorentzian geometry in low regularity has captured the interest of researchers for at least a decade. But here the term ``low-regularity'', somewhat vague as it is, applies to a number of related but distinct lines of investigation. It may refer to geometric objects with low differentiability (say, Lorentzian metrics of regularity less than $C^2$) on an otherwise smooth ($C^\infty$) background manifold, but it may also involve relaxing the degree of differentiability of the manifold itself. Some notable results of the former sort include the development of causality in spacetimes with $C^0$ metrics \cite{chrusciellowcausal1, c11-1}, the Hawking and Penrose singularity theorems for $C^{1,1}$ metrics \cite{c11-3,  c11-5, c11-6}, event horizons with low regularity \cite{area}, and a theory of causality for cone structures on manifolds \cite{conefields1,minguzzicones}. The theory of Lorentzian length spaces \cite{length2, Hau_2020, length3} is a primary example of a framework with properties analogous to those of Lorentzian geometry but without differentiability or even a manifold structure. The latter picks up a sort of abstract approach to the spacetime structure going back to the so-called \textit{causal spaces} introduced Kronheimer and Penrose, and the \textit{timelike spaces} of H. Busemann \cite{kron-pen,H1967}.

A difficulty of many of the extant generalizations of geometric results in low regularity, and one that has motivated our studying them from a different perspective, is that those broader results often involve the development of difficult strategies and new, sometimes fairly \textit{ad hoc } techniques to resolve the difficulties generated by the specific brand of low regularity adopted. For example, it is well known that in the theory of Lorentzian length spaces the definition of a geodesic does not necessarily coincide with the usual definition on manifolds; also, in this environment, the generalization of convex neighborhood in Lorentzian geometry involves a technical notion called ``localizability'' (see \cite{Length1} for details); finally, the notion of ``Ricci curvatures bounded below'' requires a number of highly non-trivial technical new ideas inspired in optimal transport theory (see, \textit{e.g.}, \cite{mondino}). Turning to the geometry of orbifolds, even the definition of a ``smooth function'' in that context is quite technical (\cite{caramello3} section 1.4). These difficulties are not only a question of acquiring greater familiarity with the pertinent structures, but can act as true obstacles in attempts to generalize some classic theorems in low-regularity situations. 

The geometry on leaf spaces, on the other hand, although it can be very irregular from a differential-geometric perspective, and often even non-Hausdorff, has a natural, relatively straightforward regularization algorithm via the transverse geometry of the foliation. This allows one to use comparatively much simpler and more systematic differential-geometric techniques, and interpret these as addressing geometric problems on the leaf space. 

This philosophy can be brought to bear on one key notion in Lorentzian geometry, that of a \textit{causal structure}. This concerns the connectivity of points in a Lorentzian manifold via the so-called \textit{causal curves}, and the many possible causal structures are classified via the so-called \textit{causal ladder (or hierarchy) of spacetimes} (see \textit{e.g.} \cite{minguzzi-sanchez} and references therein). In a recent work \cite{us} the idea was explored of endowing the leaf space $M/\f$ with a \textit{transverse causal structure}, obtained via a transverse Lorentzian metric on $(M,\mathcal{F})$. Using that, the authors have been able to describe the initial rungs of a \textit{transverse causal ladder}, obtaining several analogues of classic theorems in causality theory. 

However, a transverse analogue of the so-called \emph{stable causality} with its subtleties and multiple equivalent formulations was on that occasion deferred to future work. A key goal of this article is to partially bridge that gap by giving a description of a \textit{transverse stable causality}, by adapting the germane spacetime notions to the foliated setting. It is well-known that stable causality has a number of different characterizations in spacetimes, such as the existence of time functions, the existence of temporal functions, and a relation-theoretic connection with the so-called $K^+$ relation (see \cite{Minguzzi_2009} for extensive recent and classic references). The following list summarizes the main equivalent definitions (see Figure \ref{equivalences}) of  stable causality in spacetimes: 
\begin{figure}[h]
\centering{
\scalebox{0.8}{
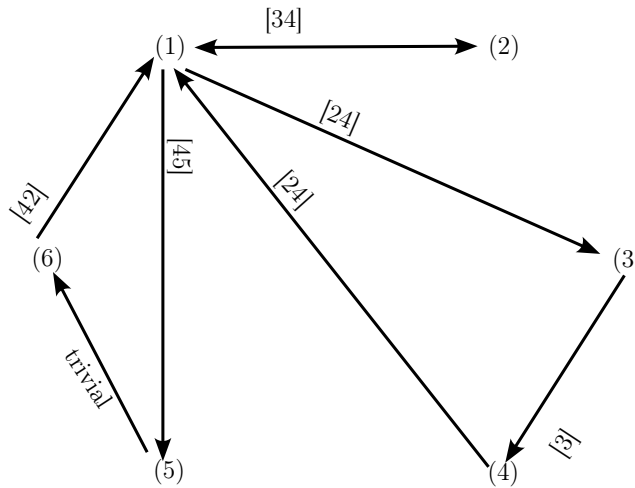}
\caption{Diagram of logical implication between the several equivalent definitions of stable causality. The reader can find proofs in the references on the corresponding arrows. }
\label{equivalences}}
\end{figure}

\begin{enumerate}
\item There is a Lorentzian metric $g^\prime$ on $M$ with $g<g^\prime$ such that $(M, g^\prime)$ is a causal spacetime, where $g<g^\prime$ means that every vector $g$-causal is $g^\prime$-timelike;
\item There is a $C^0$-topology neighborhood $\mathcal{U}$ of $g$ in the space of Lorentzian metrics of $M$ such that for every $g^\prime\in\mathcal{U}$ we have that $(M, g^\prime)$ is a causal spacetime;
\item $(M, g)$ admits a \textit{time function} $f:M\rightarrow \mathbb{R}$, that is, $f$ is continuous and increases along future-directed causal curves;
\item $(M, g)$ admits a \textit{temporal function} $f:M\rightarrow \mathbb{R}$, that is, $f$ is smooth and the gradient $\nabla f$ is a past-directed timelike vector field;
\item The \textit{Seifert relation} $J_S(g)$ is an antisymmetric relation on $M\times M$, where $(p,q)\in J_S(g)\iff q\in\bigcap_{g^\prime> g}J^+_{g^\prime}(p)$;
\item The $K^+$ relation is an antisymmetric relation on $M\times M$, where $K^+$ is defined as the smallest closed, reflexive and transitive relation which contains the relation $J^+_g$.
\end{enumerate}

The main goal of this paper is to introduce the ``transversal analogues'' of each of these notions around transverse stable causality, and provide partial logical connections between them analogous to those in spacetimes as far as possible. However, it is far from obvious that these transverse analogues should all be equivalent in the broader context we consider here; indeed, we have not been able at this juncture to establish their full equivalence in general. But we have been able to show, in the important intermediate case of foliations given by the fibers of the so-called \textit{simple} submersions, that most - though not all - of the equivalences are actually recovered. We believe that the existing relations are interesting and intriguing enough to merit further investigation, which we defer to future work. 

The rest of the paper is organized as follows. {
In {Section \ref{sec:prelim}}, we review the essential preliminary concepts of foliation theory, including basic functions, vertical distributions, and the definitions of bundle-like metrics and semi-Riemannian (and specifically Lorentzian) foliations.
	In {Section \ref{sec:firstnotion}}, we recall the basic notions of transverse causality in Lorentzian foliations first introduced in \cite{us}. These include the definitions of transversely causal vectors, transverse time orientation, and the precedence relations that form the initial rungs of the transverse causal ladder. 
	In {Section \ref{sec9}}, we introduce a core definition of \textit{transverse stable causality} which we take as the main one, and lay out its foundational properties. 
	{Section \ref{sec:whitneytopology}} discusses the application of Whitney $C^r$ topologies to define the space of transverse Lorentz metrics, setting the topological groundwork for perturbations of the metric and its relationship to the adopted transverse stable causality notion.
	In {Section \ref{subsec9.2}}, we delve into the so-called \textit{transverse time functions}, exploring the existence and implications of continuous transverse time functions and smooth \textit{transverse temporal functions}. 
	In {Section \ref{sec:k-causal}}, we investigate a notion of transverse $K$-causality, examining the transverse $K^+$ relation and discussing the nuances of extending this relation-theoretic concept to the foliated setting.
	{Section \ref{sec:submersions}} applies the framework developed in the previous sections to the important intermediate case of \textit{simple foliations} (those given by the fibers of a simple submersion), showing that in this scenario, most of the classical equivalences for stable causality are indeed successfully recovered. 
	Finally, {Section \ref{sec:causalhierq}} concludes the main body of the paper with some closing remarks on how these new definitions and equivalences fit into the broader transverse causal hierarchy. This issue turns out to be rather subtle because it involves not only causal aspects, but the topology of the leaf space must play an important role, and it remains to be clarified what the precise, optimal assumptions are for the latter. Because of these subtleties, we shall not be able deal with it in full here, but will return to it in future work. 
	
	   {Appendices \ref{appendixA} and \ref{appendixB}} expand on the many technical details regarding the $C^r$ topology for the space of transverse Lorentz metrics $\mathcal{TL}^r(M,\mathcal{F})$ and establish its homeomorphism to the space of holonomy-invariant Lorentzian fiber metrics on the normal bundle $\mathrm{Lor}^r_\intercal(\nu\mathcal F)$.}

\section{Preliminaries} \label{sec:prelim}

We fix for the remainder of this paper a \textit{regular foliated manifold}, that is, a pair $(M, \mathcal{F})$, where $M$ is a smooth (i.e., $C^\infty$) real manifold of dimension $n\geq 3$ and $\f$ is a partition of $M$ by $C^\infty$ immersed submanifolds - the \textit{leaves} of the foliation $\f$ - of codimension $q=codim(\mathcal{F})$ which arrange themselves locally ``like the fibers of a submersion'' in a precise technical sense. The reader is referred to Refs. \cite{candel, mrcun, molino} for further details. 

We briefly recall here some basic foliation-theoretic concepts we shall use throughout. A function $f:M\rightarrow\mathbb{R}$ which is constant along the leaves of $\f$ is called ($\f$-)\textit{basic}, and the set of smooth basic functions is denoted as $C^\infty(\f)$. Given any $x\in M$, let $L_x$ be the leaf of $\f$ through $x$. A vector $v \in T_xL_x$ is said to be \textit{vertical}. The \textit{vertical distribution} defined by $\f$ is simply the subbundle $T\f\subset TM$, whose fibers are the tangent spaces to the leaves, and whose space of sections we denote either by $\Gamma(T\f)$ or $\mathfrak{X}(\f)$ interchangeably. Any such (local or global) section is a \textit{vertical} vector field. Smooth basic functions are characterized by their having vanishing directional derivatives along vertical vectors. Let $\mathcal{U}\subset M$ be an open set, and denote by $T\f|_{\mathcal{U}}$ the restriction of $T\f$ to $\mathcal{U}$. A vector field $X\in\mathfrak{X}(\mathcal{U})$ is called \textit{projectable} (with respecto to $\f$) if $[V,X]\in\mathfrak{X}(\f|_{\mathcal{U}})$ for all $V\in\mathfrak{X}(\f|_{\mathcal{U}})$. (Observe that by Frobenius' theorem all vertical vector fields are immediately projectable in this sense.) The collection of all projectable vector fields forms a Lie subalgebra of $\mathfrak{X}(M)$, as well as a module over the $\mathbb{R}$-algebra $C^\infty (\f)$. 

A $(0,s)$-tensor field $T$ on $\mathcal{U}$ is \textit{horizontal} if for any $x\in \mathcal{U}$, $T_x$ vanishes whenever some vertical vertical $v\in T\f_x$ is plugged into any of its slots. $T$ is said to be \textit{holonomy-invariant} if $\mathcal{L}_VT\equiv 0$ for all $V\in\mathfrak{X}(\f|_{\mathcal{U}})$. Finally, we say that $T$ is ($\f$-)\textit{basic} if it is both horizontal and holonomy-invariant. 

As mentioned in the Introduction, the rank $q$ vector bundle on $M$ whose fiber at $x\in M$ is $T_xM/T\f_x$ is the \textit{normal bundle} $\nu\f$ of the foliation. We thus have a short exact sequence of vector bundle maps:
$$0\rightarrow T\f \rightarrow TM \rightarrow \nu \f \rightarrow 0. $$
A smooth $q$-dimensional distribution $\mathcal{H}\subset TM$ is said to be \textit{horizontal} (with respect to $\f$) if it yields a splitting of the above short sequence, i.e., we have a direct sum decomposition $TM=T\f \oplus \mathcal{H}$. 

If a horizontal distribution is chosen, then any vector $v \in \mathcal{H}$ is said to be ($\mathcal{H}$-)\textit{horizontal}. This induces a concomitant smooth decomposition of any vector field $X\in \mathfrak{X}(M)$ into a \textit{vertical part} $X^\top \in \mathfrak{X}(\f)$ and a ($\mathcal{H}$-)\textit{horizontal part} $X_H \in \Gamma(\mathcal{H})$.

A natural way to define a horizontal distribution is via a Riemannian metric $h$ on $M$: we can put $\mathcal{H}_x := T\f_x^\perp$, the set of vectors $h$-normal to $T\f_x$. This splitting naturally generalizes to the broader class of semi-Riemannian metrics via the following notion.
\begin{definition}[Bundle-like metrics]\label{def bundle-like metric}
    Let $\f$ be a codimension $q$ foliation on the manifold $M$. A semi-Riemannian metric $g$ on $M$ is \textit{bundle-like (with respect to $\f$)} if the following conditions are satisfied.
    \begin{itemize}
        \item[i)] $g$ is \textit{adapted} to $\f$, \textit{i.e.}, the rank $q$ bundle $T\f^\perp$ of vectors $g$-normal to the leaves defines a horizontal distribution. (For any $X\in \mathfrak{X}(M)$ we denote its \textit{vertical part} as $X^\top \in \mathfrak{X}(\f)$ and its ($g$-)\textit{horizontal part} as $X^\perp \in \Gamma(T\f^\perp)$.)
        \item[ii)]For every open set $\mathcal{U}\subset M$, given projectable and $g$-horizontal vector fields $X,Y\in\mathfrak{X}(\mathcal{U})$, we have that $g(X,Y)$ is a basic function on $\f|_{\mathcal{U}}$.
    \end{itemize}
\end{definition}


A key notion for us is the following.

\begin{definition}[Semi-Riemannian foliations]\label{def: SR foliations}
Let $M$ be an $n$-dimensional manifold and let $\mu,q \in \mathbb{Z}$ with $0<q<n$ and $0\leq \mu\leq q$. A {\em codimension $q$ semi-Riemannian foliation of index $\mu$} on $M$ is a pair $(\f,g_\intercal)$ such that 
\begin{itemize}
    \item[i)] $\f$ is a codimension $q$ foliation on $M$;
    \item[ii)] $g_\intercal$ is an $\f$-basic symmetric $(0,2)$-type tensor field on $M$ such that, for each $x\in M$ the bilinear form $g_\intercal(x):T_xM\times T_xM\rightarrow\mathbb{R}$ has index $\mu$ and, furthermore its kernel 
    $$\ker g_\intercal(x) := \{v \in T_xM \, : \, g_\intercal(x)(v,w)=0, \, \forall w \in T_xM \} $$
    coincides with the vertical space $T\f_x$. Such a tensor $g_\intercal$ is said to be a \textit{semi-Riemannian transverse metric (tensor)} (of index $\mu$) for $\f$. 
\end{itemize}
In particular, a semi-Riemannian foliation [resp. transverse metric] of index $1$ and codimension $q\geq2$ is called a \textit{Lorentzian foliation} [resp. \textit{Lorentzian transverse metric}]. It is a \textit{Riemannian foliation} [resp. \textit{Riemannian transverse metric}] if it has index zero. 
\end{definition}

There is a natural relationship between bundle-like metrics and transverse metrics given as follows. Given a semi-Riemannian foliation $(\f,g_\intercal)$ with index $\mu$ there always exists a bundle-like semi-Riemannian metric $g$ of the same index associated with the transverse metric $g_\intercal$: just pick any Riemannian metric $h$ on $M$ and define $g(X,Y):= g_\intercal(X^\perp, Y^\perp) + h(X^\top, Y^\top), \quad \forall X,Y \in \mathfrak{X}(M)$, where the vertical and horizontal parts on the right-hand side are taken with respect to $h$. Conversely, if $g$ is a bundle-like metric for $\mathcal{F}$, then the $(0,2)-$tensor $g_\intercal$ defined by $g_\intercal(X,Y):=g(X^\perp,Y^\perp)$ is a transverse semi-Riemannian metric on $\f$.


\section{Transverse causal structure - first notions}\label{sec:firstnotion}
We will be dealing in this paper almost exclusively with transverse Lorentzian foliations and Lorentzian bundle-like metrics thereon. By a slight abuse of notation we shall refer to the full triple $(M,\f,g_\intercal)$ where $(\f,g_\intercal)$ is a codimension $q(\geq 2)$ Lorentzian foliation on the manifold $M$ as a Lorentzian foliation, and we fix one such for the rest of this section. We shall always assume throughout that $M$ is connected. 

We shall briefly recall here the main foliation-theoretic analogues of the causal structure on spacetimes, referring the reader to \cite{us} for further details. These are indicated by the general term \textit{transverse} (or its grammatical variations) to differentiate them from the usual spacetime terms. 

Let $x\in M$. A non-vertical vector $v \in T_xM$ is said to be 
\begin{itemize}
    \item \textit{transversely timelike} if $g_\intercal(v,v)<0$,
    \item \textit{transversely lightlike} (or \textit{transversely null}) if $g_\intercal(v,v) =0$,
    \item \textit{transversely causal} if $g_\intercal(v,v) \leq 0$, i.e., if $v$ is either transversely timelike or transversely null. 
\end{itemize}
These names indicate the \textit{transverse causal character} of the vector. A vector field $X\in \mathfrak{X}(M)$ is transversely timelike [resp. transversely null, transversely causal] if $X_x$ has the corresponding causal character at each $x\in M$. It can be easily seen that set of all transversely causal vectors in $T_xM$ has two connected components therein, each of which is called a \textit{causal wedge} (at $x\in M$). A \textit{transverse time orientation} on the Lorentzian foliation $(M,\f,g_\intercal)$ is a continuous choice of a causal wedge at each point $x\in M$, called the \textit{future causal wedge} (the other component being of course the \textit{past causal wedge}). It can be shown that a transverse time orientation exists if and only if there exists a globally defined transversely timelike vector field $X\in \mathfrak{X}(M)$. For the remainder of the paper we always assume that any given Lorentzian foliation $(M, \mathcal{F}, g_\intercal)$ is a transversely time-oriented. 

A \textit{future-directed transversely timelike/null/causal} curve (or vector field) is a piecewise smooth curve in $M$ whose tangent vectors are in the future causal wedge at each of its points. A \textit{past-directed} transversely timelike/null/causal curve is defined dually. 

The \textit{transverse chronological} ($<<_{M/\f}$) and \textit{transverse causal} ($\leq_{M/\f}$) \textit{precedence relations} between leaves of $\mathcal{F}$ are defined by the existence of a future-directed transversely timelike curve [resp. transversely causal] connecting these leaves, in complete analogy with the spacetime case. Thus, if $A$ is any set saturated by the leaves of $\mathcal{F}$ (that is, it contains any lef it intersects), then we define the sets
\begin{align*}
I^+_{M/\f} (A):=& \{L\in\mathcal{F}: \exists L^\prime \subset A, L^\prime<<_\intercal L\}\\
J^+_{M/\f} (A):=& \{L\in\mathcal{F}: \exists L^\prime \subset A, L^\prime\leq_\intercal L\},\\
I^+_\intercal(A) :=& \bigcup I^+_{M/\f} (A),\\
J^+_\intercal(A) :=& \bigcup J^+_{M/\f} (A).
\end{align*}
Observe that while $I^+_{M/\f}(A)$ and $J^+_{M/\f}(A)$ are subsets of the leaf space, $I^+_\intercal(A)$ and $J^+_\intercal(A)$ are saturated subsets of $M$ itself. The sets $I^-_{M/\f}(A), J^-_{M/\f}(A), I^-_\intercal(A)$ and $J^-_\intercal(A)$ are defined analogously by ``time-duality''.

In \cite{us}, the following rungs of the \textit{transverse causal ladder} were introduced, and their relative logical strength explored: 
\begin{definition}\label{defi: transverse ladder}
We say that $(M, \mathcal{F}, g_\intercal)$ is
\begin{itemize}
\item \textit{transversely totally vicious} if $L\cap I^+_\intercal(L) \neq \emptyset$ (equivalently $L\in I^{+}_{M/\f}(L)$), $\forall L\in\mathcal{F}$;
    \item \textit{transversely chronological} if $L\cap I^+_\intercal(L) =\emptyset$ (equivalently $L\notin I^{+}_{M/\f}(L)$), $\forall L\in\mathcal{F}$;
    \item \textit{transversely causal} if $L\nless_{M/\f} L$ $ \forall L\in\mathcal{F}$;
    \item \textit{transversely future [resp. past] distinguishing} if $\forall L,L'\in \f$
    $$L\leq_{M/\f} L' \text{ and } L \in \overline{J^+_{M/\f}(L')}\text{ [resp. $L'\in \overline{J^-_{M/\f}(L)}$]} \Longrightarrow L=L',$$
    and \textit{transversely distinguishing} if it is both transversely future distinguishing and transversely past distinguishing;
    \item \textit{transversely strongly causal} if $\forall L\in \mathcal{F}$ and every saturated open neighbourhood $\mathcal{U}\supset L$ in $M$, there exists an open neighborhood $L\subset \mathcal{V}\subset \mathcal{U}$ such that if $\gamma:[a,b]\rightarrow M$ is any transversely causal curve with $\alpha(a), \alpha(b)\in\mathcal{V}$ then $\alpha([a,b])\subset\mathcal{U}$;
    \item \textit{transversely globally hyperbolic} if it is transversely strongly causal and if for all $L, L'\in\mathcal{F}$ the \textit{causal prism} $J^+_\intercal(L)\cap J^-_\intercal(L')$ is \textit{transversely compact} in $M$, meaning it projects onto a compact set in the leaf space via the standard projection $\pi_\f:M\rightarrow M/\f$. 
\end{itemize}
\end{definition}

We say that a bundle-like Lorentzian metric $g$ on $(M,\f)$ is \textit{associated with} $g_\intercal$ if $g(X^\perp,Y^\perp) = g_\intercal(X,Y), \forall X,Y \in \mathfrak{X}(M)$. We denote by $\mathcal{BL}(g_\intercal)$ the set of all bundle-like Lorentzian metrics on $(M,\mathcal{F})$ which are associated with $g_\intercal$.
The transverse notions introduced above have some logical connections with the causal structure of spacetimes produced by  Lorentzian bundle-like metrics associated to it. For example, if $(M, \f, g_\intercal)$ is transversely causal, then it is transversely chronological, and $(M,g)$ is causal for any $g\in \mathcal{BL}(g_\intercal)$. More of these relations can be seen in \cite[Sec. 5]{us}.

\subsection{Causal wedges and cone fields} An important generalization of causal structure in general, and stable causality in particular, has been given in the literature via the notion of a \textit{cone field} over a manifold \cite{ minguzzicones,Fathi, conefields1}. While these structures are fairly general, they \textit{do not} directly apply to the transverse geometry dealt with here. We present a brief discussion of the issues that arise when trying to apply the pertinent techniques. (See \cite[section 2]{Fathi} for the terminology of cone fields used in this discussion.)

Let $g_\intercal$ be a transversely time-oriented transverse Lorentzian metric on the foliated manifold $(M,\f)$. We have seen that the set $\{v \in T_x M \setminus T_x \mathcal F : g_\intercal (v,v) \leq 0\}$ consists at the two causal wedges at $x \in M.$ We denote by $\mathcal W_x$ the \emph{closed} set $\{v \in T_x M : g_\intercal (v,v) \leq 0\}$ which contains the causal wedges \emph{and} the vertical space $T\f_x$. (This is the analogue of considering the causal cones plus the zero vector in Lorentzian geometry.) Such closed cones are indeed closed convex cones in the usual sense, and have nonempty interior. One might then try and consider mappings $x \in M \mapsto \mathcal W_x \subset TM$ and study the continuity of such set-valued maps under a suitable topology. However, crucial to all the methods and results of \cite{Fathi} and other papers on closed cone fields is the requirement that the closed convex cones at each point are also \emph{complete}, that is, they does not contain a complete affine line. This is a technical assumption needed in order for the dual and polar sets of complete cones to also be complete cones (\cite{boyd}, section 2.6). This requirement \emph{fails} for causal wedges precisely because of the inclusion of the vertical spaces.  
 
Alternatively, one might try to work directly with cones on the \emph{normal bundle} $\nu\mathcal F:=TM/T\f,$ since the quotient operation maps the wedge into a cone in this bundle, but the methods in \cite{ minguzzicones,Fathi} are strictly developed for $TM$, and we were not able to adapt them to $\nu\mathcal F$. For those reasons, the cone structure methods will not be employed here.

\section{Transverse stable causality}\label{sec9}

We are ready to introduce a notion of transverse stable causality on Lorentzian foliations. While stable causality is one of the most subtle steps on the causal ladder of spacetimes, it basically means that ``small perturbations'' of the spacetime metric should not create causal loops. Our goal here is to define a transverse analogue of that notion for Lorentzian foliations by clarifying what the terms in scare quotes ought to mean in this context.

In the case of spacetimes, it is well-known that there are a number of different but equivalent ways of making precise the basic idea described in the previous paragraph: by ``opening cones'', via a suitable Whitney topology on the space of metrics, via time functions, via temporal functions, among others. We briefly review some of these concepts here, and recall how they relate to the stable causality of spacetimes, referring to the standard references \cite{beem,hawking-ellis,minguzzi-sanchez} for proofs and further discussion. Since our goal is to find transverse analogues of these notions on Lorentzian foliations, we do not describe them all at once; rather we will gradually bring them up and propose at each step a suitable transverse analogue. At the end, we shall summarize the mutual logical relationships among the various notions. 

The most convenient way of describing (spacetime) causal stability for our purposes here is as follows. First, recall we can define, on the set $\mathrm{Lor}(M)$ of (smooth) Lorentzian metrics on a fixed manifold $M$, a (strict) partial order $<$ given by 
$$\forall g,g'\in \mathrm{Lor}(M),\, g<g' \text{ if for any $v\in TM\setminus \{0\}$, } g(v,v)\leq 0 \Rightarrow g'(v,v)<0.$$
In words, we write $g<g'$ if and only if any $g$-causal vector is $g'$-timelike, or equivalently, iff for all $x\in M$, the $g$-causal cone in $T_xM$ is contained in the interior of the $g'$-causal cone, i.e., the $g'$-timecone at $x$. Informally, one often describes this by saying that that ``$g$ has narrower cones than $g'$'', or that the cones of $g'$ are ``more open'' than those of $g$. 

With this definition in place, we say that a spacetime $(M,g)$ is \textit{stably causal} if there exists some $g' \in \mathrm{Lor}(M)$ such that $(M,g')$ is a causal spacetime and $g<g'$. In other words, we can ``open up the cones a little'' without creating causal loops. 

This definition is particularly convenient for us because it has an obvious transverse analogue. Given a foliated manifold $(M,\mathcal{F})$ of codimension $\geq 2$, we can define, on the set $\mathcal{TL}(M,\mathcal{F})$ of \textit{transverse} Lorentzian metrics on $(M,\mathcal{F})$ the (strict) partial ordering
$$\forall g_\intercal,g'_\intercal \in \mathcal{TL}(M,\mathcal{F}),  g_\intercal<g'_\intercal \text{ if for any $v\in TM\setminus T\mathcal{F}$, } g_\intercal(v,v)\leq 0 \Rightarrow g'_\intercal(v,v)<0.$$

We can now define our ``fiducial'' notion of transverse stable causality. 
\begin{definition}[Transverse stable causality]\label{trstablecsldefi}
A (transversely time-oriented) Lorentzian foliation $(M,\mathcal{F},g_\intercal)$ is \emph{transversely stably causal} if there exists some $g'_\intercal \in \mathcal{TL}(M,\mathcal{F})$ such that $g_\intercal <g'_\intercal$ and for which $(M,\mathcal{F},g'_\intercal)$ is a transversely causal Lorentzian foliation. 
\end{definition}

\begin{rmk}\label{rmkstablycausalimpliescausal}
 \emph{It is clear from the previous definition that any transversely stably causal Lorentzian foliation is also transversely causal. Moreover, when $g_\intercal <g'_\intercal$ a transverse time-orientation for $(M,\mathcal{F},g'_\intercal)$ is always fixed by declaring its future wedges to be those that contain the future wedges of $g_\intercal$, and we assume this convention throughout without further comment. }   
\end{rmk}

We define here a class of examples which will be further explored in Section \ref{sec:submersions}.
\begin{exe}[Lorentzian submersions]\label{submersionsexe}
 {\em 
Let $M^{n+q},B^q$ with $q\geq 2$ be connected manifolds, and let $\pi:M\rightarrow B$ be an onto submersion. Take $\f$ to be the foliation given by the connected components of the fibers of $\pi$. Let $h$ be a (time-oriented) Lorentz metric on $B$, and assume that $(B,h)$ is a stably causal spacetime. We easily check that $g_\intercal:= \pi^\ast h$ is a transverse Lorentzian metric on $(M, \f)$. Moreover, $(M,\f, g_\intercal)$ becomes a transversely time-oriented Lorentzian foliation if we pick at each point $x\in M$ as the future causal wedge the one that projects through $\pi$ onto the future causal cone at $\pi(x) \in B$. 
 
 Now, pick another Lorentz metric $h'$ on $B$ so that $h<h'$ and $(B,h')$ is a causal spacetime, and let $g'_\intercal: = \pi^\ast h'$ be the associated transverse Lorentzian metric. We first claim that $g_\intercal <g'_\intercal$. Indeed, given $x\in M$ and $v\in T_xM$ $g_\intercal$-transversely causal, we have 
 $$h_{\pi(x)}(d\pi_x(v),d\pi_x(v)) \leq 0 \stackrel{h<h'}{\Longrightarrow} h'_{\pi(x)}(d\pi_x(v),d\pi_x(v))(\equiv (g'_\intercal)_x(v,v))<0,$$
 thus proving the claim. 

 Next, we claim that $(M,\f,g'_\intercal)$ is a transversely causal foliation. Suppose not, and let $\gamma:[0,1] \rightarrow M$ be a future-directed transversely causal curve (with respect to $g'_\intercal$) such $\gamma(0),\gamma(1)$ are on the same leaf $L\subset \pi^{-1}(b)$ with $b \in B$. Then $\pi\circ \gamma$ is a closed $h'$-causal curve at $b$, contradicting the fact that $(B,h')$ is causal. We conclude that $(M,\f,g_\intercal)$ is a transversely stably causal foliation.
 }   
\end{exe}



\begin{exe}[Pullbacks]\label{pullbackssexe}
 {\em 
Let $(N,\g,h_\intercal)$ be a transversely time-orientable Lorentzian foliation, and consider a smooth map $\phi:M\rightarrow N$ which is transverse to the foliation $\g$ in the sense that 
 $$T_{\phi(x)}N = d\phi_x(T_xM)+ T_{\phi(x)}L_{\phi(x)}, \quad \forall x\in M,$$
 where $L_{\phi(x)}\subset N$ denotes the leaf of $\g$ through $\phi(x)$. The \textit{pullback foliation} $\f:=\phi^\ast\g$ on $M$ is given by the connected components of the sets $\phi^{-1}(L)$ for $L\in \g$. One can easily check that (1) $g_\intercal := \phi^\ast h_\intercal$ is a transverse Lorentzian metric on $(M,\f)$ such that $v\in T_xM$ is transversely null (resp. transversely timelike) on $(M,\f,g_\intercal)$ if and only if $d\phi_x(v)$ is transversely null (resp. transversely timelike) on $(N,\g,h_\intercal)$, (2) a natural transverse time-orientation is given on $(M,\f,g_\intercal)$ by declaring, for each $x\in M$, that $d\phi_x$ maps the future causal wedge of the latter foliation at $x$ into that of $(N,\g,h_\intercal)$ at $\phi(x)$. We refer to this as the induced transverse time-orientation. 

 Assuming $(N,\g,h_\intercal)$ is transversely stably causal, we claim that the pullback foliation $(M,\f,g_\intercal)$ with the induced transverse time-orientation is also transversely stably causal\footnote{The analogous statement also holds for the transverse chronological condition, the transverse causal condition, and transversely strong causality, with a proof straightforwardly adapted from the one we give here.}. Indeed, let $h'_\intercal$ be a transverse Lorentzian metric on $(N,\g)$ such that $h_\intercal<h'_\intercal$ and $(N,\g,h'_\intercal)$ is transversely causal. It is thus clear that $g_\intercal <\phi^\ast h'_\intercal =:g'_\intercal$. There remains to show that $(M,\f,g'_\intercal)$ is transversely causal. But given a transversely causal curve $\gamma:[0,1] \rightarrow M$ for $(M,\f,g'_\intercal)$ starting and ending at the same leaf $L \in \f$, $\phi\circ \gamma$ will be a transversely causal curve for $(N,\g,h'_\intercal)$ starting and ending at the same leaf $L'\supset \phi(L)$ of $\g$, in contradiction with the transverse causal condition therein. 
 }
\end{exe}

\medskip 

It is convenient to try and relate the notion of transverse stable causality with the (standard) causal stability of associated bundle-like metrics. 


\begin{teo}\label{widen-the-cone-widen-the-wedge}
Given a foliated manifold $(M,\mathcal{F})$ and $g_\intercal,g'_\intercal \in \mathcal{TL}(M,\mathcal{F})$, the following statements are equivalent.
\begin{enumerate}
    \item $g_\intercal<g_\intercal^\prime$.
    \item $\forall g\in\mathcal{BL}(g_\intercal), \exists g^\prime\in\mathcal{BL}(g^\prime_\intercal)$ such that $g<g^\prime$.
   \item $\exists g\in\mathcal{BL}(g_\intercal),\exists g^\prime\in\mathcal{BL}(g^\prime_\intercal)$ such that $g<g^\prime$.
\end{enumerate}
\end{teo}
\begin{proof}
$(1)\implies(2)$ \\
Suppose $g_\intercal<g_\intercal^\prime$ and let $g\in\mathcal{BL}(g_\intercal)$. It is not difficult to check that there exists a Riemannian metric $h$ on $M$ such that for all $v\in TM$ we have 
\begin{equation}\label{usolocal}   
g(v,v)=g_\intercal(v,v)+h(v^\top,v^\top),
\end{equation}
where $v^\top$ is the component of $v$ along $T\mathcal{F}$ in the decomposition of $TM$ produced by $h$. 

Now, pick a locally finite open covering $\mathcal{C}=\{U_i\}_{i\in I}$ of $M$ by precompact open sets, and consider for each $i\in I$ the compact set $\mathcal{K}_i \subset TM$ given by 
$$\mathcal{K}_i := \{v\in T_pM \, :\, p \in \overline{U}_i, h_p(v,v)=1, (g_{\intercal})_p(v,v)\leq 0\}.$$
The function $\Phi: v\in TM \mapsto g'_\intercal(v,v) \in \mathbb{R}$ is obviously continuous, and hence it attains a minimum 
$\varepsilon_i := \min\{g_\intercal(v,v) \, : \, v\in \mathcal{K}_i\}$
for each $i\in I$. Since $\Phi|_{\mathcal{K}_i}$ is strictly positive because $g_\intercal<g_\intercal^\prime$, we have $0< \varepsilon_i$ for each $i\in I$. Finally, fix a $\{\varphi_i\}_{i\in I}$ a smooth partition of unity subordinate to $\mathcal{C}$. 

Next, for each $i \in I$, define $f_i \in C^\infty(M)$ by 
$$f_i(p) := \left\{ \begin{array}{cc}
   \varepsilon_i \varphi_i(p),  & \text{if $p \in U_i$} \\
    0, & \text{if $p\notin supp \, \varphi_i $}
\end{array}\right. ,$$
and $f\in C^\infty(M)$ by 
$$f := \sum_{i\in I} f_i.$$
For each $p\in M$ define the number $$
\alpha_p:=\inf\{-g^\prime_{\intercal,p}(v,v)| v\in T_pM \text{ satisfies } h_p(v,v)=1, g_{\intercal,p}(v,v)\leq 0\}.
$$

We claim that $0< f(q) \leq \alpha_q$ for every $q\in M$. To see this, fix $p\in M$, and let $V\ni p$ be an open set for which the set 
$J:= \{i \in I \, : \, U_i \cap V \neq \emptyset\}$
is finite. Given any $v \in T_pM \text{ such that } h_p(v,v)=1, (g_{\intercal})_p(v,v)\leq 0$, we have
$$f(p) =\sum_{i\in J}\varepsilon_i \varphi_i(p) \leq \left (\sum _{i\in J}\varphi_i(p)\right)(g_\intercal)_p(v,v) \leq (g_\intercal)_p(v,v),$$
whence the claim follows. 

Define $g^\prime\in\mathcal{BL}(g^\prime_\intercal)$ by 
$$
g^\prime(v,w)=g^\prime_\intercal(v,w)+ \frac{f}{2}\cdot h(v^\top, w^\top), \forall v,w\in TM.
$$ 
To see that $g^\prime$ indeed works, let $v\in T_pM\setminus \{0\}$ be such that $g(v,v)\leq 0$. Then, $g_\intercal(v,v)\leq 0$, and we compute 
\begin{align*}
    g^\prime(v,v)&=g^\prime_\intercal(v,v)+\frac{f(p)}{2} h(v^\top, v^\top)\\
    &=\left[g'_\intercal\left(\frac{v}{\|v\|_h},\frac{v}{\|v\|_h}\right)+\frac{f(p)}{2} h\left(\frac{v^\top}{\|v\|_h}, \frac{v^\top}{\|v\|_h}\right)\right] \|v\|^2_{h}\\
    &\leq \left[g'_\intercal\left(\frac{v}{\|v\|_h},\frac{v}{\|v\|_h}\right)-\frac{1}{2} g'_\intercal\left(\frac{v}{\|v\|_h}, \frac{v}{\|v\|_h}\right)\right] \|v\|^2_{h}\\
    &= \frac{1}{2}g^\prime_\intercal\left(\frac{v}{||v||_h}, \frac{v}{||v||_h}\right)h(v, v) <0.  
\end{align*}
$(2)\implies(3)$ is immediate.\\

\noindent $(3)\implies(1)$ \\
Let $ g\in\mathcal{BL}(g_\intercal),g^\prime\in\mathcal{BL}(g^\prime_\intercal)$ be given such that $g<g^\prime$. Fix $p\in M$ and a nonvertical vector $v\in T_pM$ such that $g_\intercal(v,v)<0$. 

Write $v=v^\top+ v^\perp$ for the decomposition of $v$ into a vertical part $v^\top$ and horizontal part $v^\perp$ with respect to $g$. We then have $g(v^\perp,v^\perp) \equiv g_\intercal(v,v)\leq 0$, and by our assumption we have then $g'(v^\perp,v^\perp) <0$. Now, $g '(v^\perp,v^\perp)$ and $g'_\intercal(v^\perp,v^\perp)$ are related by an equation analogous to \eqref{usolocal}, whence it follows that $g'_\intercal(v^\perp,v^\perp)<0$; However, the $\mathcal{F}$-basic nature of $g'_\intercal$ means that 
$$g'_\intercal(v,v) = g'_\intercal(v^\perp,v^\perp)<0,$$
and we conclude that $g_\intercal < g'_\intercal$, as desired. 
\end{proof}

\begin{corollary}\label{cortansvstableimpliesstable}
If the Lorentzian foliation $(M,\mathcal{F},g_\intercal)$ is transversely stably causal, then for any $g\in\mathcal{BL}(g_\intercal)$ there exists a Lorentzian bundle-like metric $g'$ on $(M,\mathcal{F})$ such that $g<g'$, with the associated transverse metric $g'_\intercal$ the foliation $(M,\f,g'_\intercal)$ is transversely causal, and $(M,g')$ is a causal spacetime. In particular, $(M,g)$ is a stably causal spacetime. 
\end{corollary}
\qcd
\begin{rmk}\label{quitecutermk}
    {\em It is not difficult to construct a stably causal spacetime containing a Lorentzian foliation which is not transversely stably causal, so the converse of the previous corollary is false in general. Let 
    $$M:=\{(t,x,y)\in \mathbb{R}^3\, : \, x^2+y^2>1\}$$
    with the flat metric $g= -dt^2+dx^2+dy^2$. A time orientation is picked such that $\partial_t$ is future-directed. The spacetime $(M,g)$ thus defined is clearly stably causal, since, say, $f: (t,x,y) \in M \mapsto t \in \mathbb{R}$ is a temporal function thereon. The vector field 
    $$K= x\partial_y-y\partial _x + \partial_t$$
    is Killing, complete and spacelike, and hence generates an isometric $\mathbb{R}$-action with spacelike orbits on $(M,g)$. These orbits are leaves of a foliation $\f$ with respect to which $g$ is bundle-like. (See Example \ref{homogeneousexe} below.) Hence it naturally defines a Lorentzian foliation which is not even transversely chronological: the timelike curve $\gamma(s)=(s,2,0)$ intersects the leaf 
    $$L=\{(\lambda,2\cos \lambda,2\sin \lambda): \lambda \in \mathbb{R}\}$$
    whenever $s=\lambda = 2\pi k$ with $k \in \mathbb{Z}$. 
    } 
\end{rmk}
\section{Whitney topologies and the space of transverse Lorentz metrics}\label{sec:whitneytopology}
Another standard characterization of stable causality on spacetimes is given via the (strong) Whitney topologies on the space of Lorentzian metrics on a manifold. Specifically, recall that a spacetime $(M,g)$ is stably causal if and only if there exists an open neighborhood $U(g)$ of $g$ in the space of Lorentzian metrics on $M$ with the Whitney $C^0$ topology so that for any $g' \in U(g)$ the spacetime $(M,g')$ is causal. 

In this section we study an analogously defined space of all possible \emph{transverse} Lorentzian metrics on a given foliated manifold $(M, \mathcal F).$  Many different topologies may be chosen to study spaces of metrics according to the geometric or analytic context one may be interested in, but in the standard Lorentzian setting the strong Whitney topologies have been shown to be particularly useful, both for mathematical and physical reasons (see \cite{lerner,lerner72}). Therefore, we will adapt those topologies here to define the space of transverse Lorentz metrics as a topological space.  

Denote by $\mathcal{TL}^r(M,\mathcal{F})$ the set of all \textit{transverse} Lorentzian metrics of class $C^r$, and by $\mathcal{BL}^r(M,\f)$ the set of all Lorentzian \textit{bundle-like} metrics on $(M,\mathcal{F})$ of class $C^r$, (with $1 \leq r \leq \infty$) on $(M,\mathcal{F})$. One is usually interested in $r = \infty,$ and in that case we shall drop the superscript and denote it simply as $\mathcal{TL}(M,\mathcal{F})$. We leave some of the more technical discussion regarding Whitney topologies for $\mathcal{TL}^r(M,\mathcal{F})$ to appendix \ref{appendixA}, where in particular we prove the following:

\begin{proposition}\label{small}
	For each $1\leq r\leq\infty$ the set $\mathcal{TL}^r(M,\mathcal{F})$ of transverse Lorentzian metrics on $(M,\f)$ of class $C^r$ is a Baire space in the strong Whitney $C^r$ topology.
\end{proposition}
\qcd

We can also present an alternative but equivalent way of describing the space of transverse Lorentzian metrics that shall prove relevant when discussing the relation between transverse causal stability as given in Definition \ref{trstablecsldefi} and another notion of causal stability understood in the sense of open subsets of the space of transverse metrics preserving transverse causality as in subsection \ref{sec:intervaltop} below. 

Consider the normal bundle $\nu\f=TM/ T\mathcal F$ of the foliated manifold $(M,\mathcal F),$ and let $\mathrm{Sym}^{(0,2)}(\nu\mathcal F)$ denote the fiberwise $(0,2)$ symmetric tensors over $\nu\f.$   Following \cite[section 2.5]{us},   given $B \in \Gamma^r(\mathrm{Sym}^{(0,2)}(\nu\mathcal F))$ we can define its Lie derivative in the direction of $V \in \mathfrak{X}(\mathcal F)$ as 
\[
\mathcal L_V B(\overline X, \overline Y) =  V(B(\overline X, \overline Y)) -  B(\mathcal L_V \overline X, \overline Y) - B(\overline X, \mathcal L_V \overline  Y), \quad \overline X, \overline Y \in \Gamma^r(\nu\mathcal F),
\]
where $\mathcal L_V \overline X = \overline{[V,X]}.$ We define the \textit{holonomy-invariant $(0,2)$-symmetric $C^r$ tensor fields of} $\nu\mathcal F$ as those sections $B \in \Gamma^r(\mathrm{Sym}^{(0,2)}(\nu\mathcal F))$ satisfying $\mathcal L_V B = 0$ for all $V \in \mathfrak{X}(\mathcal F).$ Denoting $\mathrm{Lor}^r_\intercal(\nu\mathcal F)\subset\Gamma^r(\mathrm{Sym}^{(0,2)}(\nu\mathcal F)) $ the subset of $C^r$ holonomy-invariant Lorenzian fiber metrics of $\nu\mathcal F$, we show in Appendix \ref{appendixB}:
\begin{proposition}\label{smallprime}
$\mathcal{TL}^r(M,\mathcal{F})$ is homeomorphic to $\mathrm{Lor}^r_\intercal(\nu\mathcal F)$ in their respective strong Whitney $C^r$ topologies.	
\end{proposition}
\qcd

%
%

 \subsection{Interval topology on the set of Lorentz metrics on vector bundles}\label{sec:intervaltop}

When considering the causal theory on spacetimes, it is standard to analyze the \textit{conformal class} of metrics instead of the metrics themselves, since the causal structure is conformally invariant. Notice that the relation $g < h$ is unaltered by conformal scalings of the Lorentzian metrics $g$ and $h$ on a manifold $M$. Denoting by $\bm{g}$ the conformal class of a metric $g,$ the strict partial order $\bm{g} < \bm{h}$ is well-defined.
An \emph{interval} of conformal metrics is the set 
\begin{equation}\label{intervals2}
	(\bm g, \bm h) = \{ \bm k : \bm g < \bm k < \bm h  \}. 
\end{equation} 
The \emph{interval topology} is a topology on the set $\mathrm{Con}(M)$ of conformal Lorentzian metrics of $M$ generated by considering all intervals \eqref{intervals2} as a subbasis. Surprisingly \cite{beem}, the interval topology coincides with the quotient topology on $\mathrm{Con}(M)$ induced by the projection $g \in \mathrm{Lor}(M) \mapsto \bm{g} \in \mathrm{Con}(M)$ where $\mathrm{Lor}(M)$ is the set of smooth Lorentzian metrics on $M$ endowed with the $C^0$ Whitney topology. (For simplicity, we shall also refer to that quotient topology as the $C^0$ topology on $\mathrm{Con}(M)$ if there is no risk of confusion.) Thus, $(M,g)$ is causally stable if and only if there exists a $C^0$ open set $U \ni \bm g$ such that if $\bm{h} \in U,$ then $(M,h)$ is causal.

We want to consider an analogous development for Lorentzian fiber metrics in $\mathrm{Lor}(\nu\mathcal F),$  the subspace of holonomy-invariant metrics $\mathrm{Lor}_\intercal(\nu\mathcal F)\subset\mathrm{Lor}(\nu\mathcal F)$ and the conformal classes for these spaces. The equivalence of interval topology and $C^0$ topology is well described in \cite{ lerner, lerner72} for Lorentzian metrics on $M,$  and by inspecting those papers, it turns out that little effort is needed to adapt Lerner's techniques to the space of Lorentzian fiber metrics over a vector bundle $E \to M,$ the standard case being recovered by setting $E = TM$; hence, we review of these ideas here.

Let $\pi: E \to M$ be a smooth vector bundle, and $ \widehat \pi :\widehat E \to M$ the fiber bundle originated by removing the image of the zero section of $E.$  At each fiber of $\widehat E$  we consider the \textit{conformal equivalence:} $v, w \in \widehat E_x$ are equivalent iff $v = \lambda w$ for some $\lambda >0.$ By considering these fiberwise relations on local trivializations of $\widehat E,$  one constructs  the \emph{conformal bundle} of $E,$ denoted by $E^C,$ with class projection $\mathcal Q : \widehat E \to E^C$ being a fiber-preserving smooth surjective submersion   (c.f. \cite[p. 21]{lerner}, \cite[pp 11-12]{lerner72}). If $E$ has rank $k,$ the fibers of $E^C$ are diffeomorphic to the sphere $\mathbb{S}^{k-1}.$ 
If $\sigma$ is a section of $\widehat E,$ we denote the section $\mathcal Q \circ \sigma$ of $E^C$ by $\bm \sigma.$  We have the following properties (p. 12, corollaries 1 and 2 in \cite{lerner72}):

\begin{proposition}\label{prop:conformalprop}
	If $\sigma, \eta$ are two $C^r$ sections of $\widehat E,$ then $\bm \sigma = \bm \eta$ iff there exists $f \in C^r(M)$ positive such that $\sigma = f \eta.$ Also, any $C^r$ section of $E^C$   can be lifted to a $C^r$ section of $\widehat E,$ that is, if $\omega \in \Gamma^r(E^C),$ there exists $\sigma \in \Gamma^r(\widehat E)$ with $\mathcal Q \circ \sigma = \omega.$ 
\end{proposition}

If $F$ is any subbundle of $E$ disjoint from the zero section, we can naturally adapt the conformal construction done above for $F$, and $F^C$ will have a fiber diffeomorphic to some submanifold of a sphere.  Here we will be interested in the Lorentzian bundle $\mathcal N_1(E)$ of fiberwise $(0,2)$-symmetric nondegenerate tensors of index one on $E$, whose space of class $C^r$ sections we denote by $\mathrm{Lor}^r(E)$ ($\mathrm{Lor}(E)$ for $r = \infty$). The former is an open subbundle and disjoint from the image of the zero section of $\mathrm{Sym}^{(0,2)}(E)$. We can then consider its conformal bundle $\mathcal N_1(E)^C,$ which we will call the \emph{conformal Lorentzian bundle} of $E$ and we will denote its space of class $C^r$ sections by $\mathrm{Con}^r(E)$ (again, $\mathrm{Con}(E)$ for $r = \infty$).



The $C^0$ topology is the most relevant for causal stability. Proposition 2.4 in \cite{lerner} can be straightforwardly adapted to obtain the following.

\begin{proposition}
	$\mathcal Q_\# : \mathrm{Lor}^r(E) \to \mathrm{Con}^r(E)$ is a continuous surjection. When considering $C^0$ topologies, it is a quotient mapping.
\end{proposition}

The notion of causal cone widening and intervals of Lorentzian fiber metrics is completely analogous to the discussion at the start of this section. To define the interval topology, one must require all intervals to be nonempty. This is done in \cite{lerner72}, pp. 29 - 30, where even more is obtained: the interval $(g,h)$ is a convex cone in $\mathrm{Lor}^r(E)$.

\begin{proposition}
	Let $g,h \in \mathrm{Lor}^r(E)$ with $\bm g < \bm h.$  For $\alpha, \beta >0$ we have $\alpha g + \beta h \in  \mathrm{Lor}^r(E) $ and  $\bm g < \bm{\alpha g+ \beta h} < \bm h,$  where $\bm{\alpha g+ \beta h}$ is the conformal class of $\alpha g + \beta h$. Moreover if $\bm l, \bm k \in (\bm g, \bm h)$, choosing representatives  $k \in \bm k$ and $l \in \bm l$,  then $\alpha l + \beta k \in  \mathrm{Lor}^r(E)$ and $\bm g < \bm{\alpha l+ \beta k} < \bm h$. 
\end{proposition}

We state the main theorem equating $C^0$ and interval topologies. 

\begin{teo}\label{teo:c0equalinterval}
	On $\mathrm{Con}^r(E)$ the $C^0$ topology and the interval topology coincide. 
\end{teo}

The proof of \ref{teo:c0equalinterval} in \cite[theorem 4.4]{lerner72} for $\mathrm{Lor}(M)$ is somewhat elaborated, but the adaptation to the more general $\mathrm{Lor}(E)$ is straightforward. The philosophy of this proof is to establish properties at a fiber based on a point $x \in M$, then by continuity, such properties remain valid for a neighborhood of $x.$ To globalize these properties, a careful patching with partitions of unity is then used.

Returning to our main analysis, we now have the equivalence of the interval topology and the $C^0$ topology for $\mathrm{Lor}^r(\nu\mathcal F))$ by theorem \ref{teo:c0equalinterval}, which we use to prove that transverse causal stability implies stability in the open sense. 

\begin{corollary}\label{widenwedgesimplyc0}
If $(M, \mathcal{F}, g_\intercal)$ is transversely stably causal, then there exists a $C^0$-neighborhood $\mathcal{U}$ of $g_\intercal$ on $\mathcal{TL}(M,\f)$ such that $(M, \f, g^\prime_\intercal)$ is transversely causal for every $g^\prime_\intercal\in\mathcal{U}$.
\end{corollary}
\begin{proof}
    Since $\mathrm{Lor}_\intercal (\nu \mathcal F)$ is homeomorphic to $\mathcal{TL}(M,\f),$ we work with $\mathrm{Lor}_\intercal (\nu \mathcal F).$ Given $g \in \mathrm{Lor}_\intercal (\nu \mathcal F)$ stably causal, there exists $g_0 \in \mathrm{Lor}(\nu \mathcal F)$ causal with  $ g < g_0.$ By the equivalence of the interval and $C^0$ topology, $\bm g$ is in some interval $\mathcal U=(\bm h, \bm{g}_0), \subseteq \mathrm{Con}(\nu \mathcal F)$ which is a $C^0$-open set. Then, setting $\mathcal V =\mathcal Q_{\#}^{-1}(\mathcal U),$ $\mathcal V$ is an open subset of $\mathrm{Lor}(\nu\mathcal F)$ with $k < g_0$ for all $k\in \mathcal V.$  The set $\widetilde{\mathcal V} = \mathcal V \cap \mathrm{Lor}_\intercal(\nu\mathcal F)$ is then $C^0$-open for the subspace topology of $\mathrm{Lor}_\intercal(\nu\mathcal F)$, it contains $g$, and is such that all holonomy-invariant Lorentzian fiber metrics $h \in \mathcal V$ are causal since $h < g_0.$
\end{proof}
\color{blue}
\color{black}
\section{Transverse time functions}\label{subsec9.2}

Recall that a \textit{time function} on a spacetime $(M,g)$ is a continuous function $f:M\rightarrow \mathbb{R}$ such that for any future-directed causal curve $\alpha:[a,b] \rightarrow M$ we have
$$\forall t,s \in [a,b], \; t<s \Longrightarrow f\circ \alpha(t) < f\circ \alpha(s),$$
or equivalently,
$$\forall p,q \in M, \; p<q \Longrightarrow f(p)<f(q).$$
As its name suggests, the physical importance of time functions is that they give a global notion of ``time elapsed between events'' in geometric theories of gravity such as General Relativity, and its level hypersurfaces model a notion of ``simultaneity hypersurface''. It is well-known that time functions exists on a spacetime $(M,g)$ if and only if the latter is stably causal. 

The transverse analogue of a time function for a Lorentzian foliation is as follows.
\begin{definition}[Transverse time function]\label{transvtimedefi} 
A \emph{transverse time function} on $(M,\mathcal{F},g_\intercal)$ is a continuous function $f:M\rightarrow \mathbb{R}$ such that for any future-directed transversely causal curve $\alpha:[a,b] \rightarrow M$ we have
$$\forall t,s \in [a,b], \; t<s \Longrightarrow f\circ \alpha(t) < f\circ \alpha(s).$$
\end{definition}
One might wonder whether there is any definite relationship between transverse time functions and ordinary time functions. The next result gives a precise answer.
\begin{proposition}\label{proptransvtimevstime}
 Let $f:M\rightarrow \mathbb{R}$ be a continuous function. The following statements are equivalent.
 \begin{itemize}
     \item[i)] $f$ is a transverse time function on $(M,\mathcal{F},g_\intercal)$.
     \item[ii)] $f$ is \emph{basic}, i.e., constant on the leaves of $\mathcal{F}$, and for any bundle-like Lorentzian metric $g$ on $(M,\mathcal{F})$ associated with $g_\intercal$, $f$ is a time function on $(M,g)$.
     \item[iii)] $f$ is basic and there exists a bundle-like Lorentzian metric $g$ on $(M,\mathcal{F})$ associated with $g_\intercal$ such that $f$ is a time function on $(M,g)$. 
 \end{itemize}
\end{proposition}
\begin{proof}
    $(i) \Rightarrow (ii)$ \\
    Suppose $f:M\rightarrow \mathbb{R}$ is a transverse time function on $(M,\mathcal{F},g_\intercal)$, and fix an arbitrary Lorentzian bundle-like metric $g$ associated with $g_\intercal$. Since any $g$-causal vector is transversely causal, $f$ is a time function on the spacetime $(M,g)$. It remains to show it is indeed constant on the leaves of $\mathcal{F}$. 
    
    Suppose, by contradiction, it is not. In that case, there is some leaf $L\in \mathcal{F}$ and points $p\neq q \in L$ so that $f(p)< f(q)$. Pick any number $\varepsilon >0$ so that $f(p)+\varepsilon < f(q)$, and using the Hausdorffness of $M$ and continuity of $f$, choose disjoint open sets $U\ni p$ and $V\ni q$ in $M$ so that $f(p')< f(p)+\varepsilon$ for each $p'\in U$. Finally, choose any future-directed $g$-causal curve $\alpha:[0,1] \rightarrow U$. By the Causal Waterfall Lemma (see \cite{us}, Lemma 4.19), there exists a future-directed transversely causal curve $\beta:[0,1] \rightarrow M$ so that $\beta(0) =q$ and $\beta(1) = \alpha(1)$. But then
    $$f(p) < f(q) = f(\beta(0)) \stackrel{f \text{ is transverse time}}{<} f(\beta(1))= f(\alpha(1)) < f(p) + \varepsilon, $$
    an absurd. We conclude that $f$ is indeed basic.\\
    $(ii)\Rightarrow (iii)$ is trivial.\\
    $(iii)\Rightarrow (i)$\\
    Suppose there exists a Lorentzian bundle-like metric $g$ on $(M,\mathcal{F})$ associated with $g_\intercal$ for which $f$ is a time function, and assume that $f$ is constant on the leaves of $\mathcal{F}$. Let $\alpha:[a,b]\rightarrow M$ be any future-directed transversely causal curve. Let $\mathcal{H}= \{(U_i,\pi_i,\gamma_{ij})\}_{i,j\in I}$ be a Haefliger cocycle of $(M,\mathcal{F})$. The compactness of the image $\alpha[a,b]$ allows one to cover it by a finite number, say $U_{i_1}, \dots, U_{i_k}$ of the domain (simple) sets in $\mathcal{H}$, and fix a partition $t_0=a<t_1 <\cdots < t_k=b$ so that $\alpha[t_{\ell-1},t_\ell] \subset U_{i_\ell}$ for each $\ell=1,\ldots, k$. In addition, for each such $\ell$ there exists a unique Lorentzian metric $h_{\ell}$ on $V_{\ell}:=\pi_{i_\ell}(U_{i_\ell})\subset \mathbb{R}^d$ (where $d$ the codimension of $(M,\mathcal{F})$) such that $\pi_\ell:= \pi_{i_{\ell}} :(U_{i_\ell}, g_{U_{i_\ell}}) \rightarrow (V_\ell,h_\ell)$ is an onto Lorentzian submersion. Furthermore,
    \begin{eqnarray}
        h_\ell(\pi_\ell\circ \dot{\alpha}|_{[t_{\ell-1},t_\ell]}, \pi_\ell \circ \dot{\alpha}|_{[t_{\ell-1},t_\ell]}) &=& (\pi_\ell^\ast h_\ell)(\dot{\alpha}|_{[t_{\ell-1},t_\ell]}, \pi_\ell \circ \dot{\alpha}|_{[t_{\ell-1},t_\ell]})\nonumber \\
&\equiv& g_\intercal(\dot{\alpha}|_{[t_{\ell-1},t_\ell]}, \pi_\ell \circ \dot{\alpha}|_{[t_{\ell-1},t_\ell]})\leq 0,\nonumber
\end{eqnarray}
    so $\alpha|_{[t_{\ell-1},t_\ell]}$ is $h_\ell$-causal in $V_\ell$. By the Causal Waterfall Lemma (see \cite{us}, Lemma 4.19), we can then choose a $g$-causal curve $\beta_\ell:[t_{\ell-1},t_\ell]\rightarrow U_{i_\ell}$ such that $L_{\alpha(t_{{\ell}-1})}=L_{\beta_\ell(t_{\ell -1})} $ and $L_{\alpha(t_{{\ell}})}=L_{\beta_\ell(t_{\ell})} $. Therefore, since $f$ is basic,
    $$f(\alpha(t_{{\ell}-1})= f(\beta(t_{{\ell}-1}) \stackrel{\text{$f$ is time function}}{<}f(\beta(t_{{\ell}})=f(\alpha(t_{{\ell}}).$$
    We then conclude that $f(\alpha(a))<f(\alpha(b))$, so that $f$ is indeed a transverse time function, thus completing the proof.
\end{proof}

Before we are ready to state and prove the main result of this section, we shall need a lemma of independent interest. 
\begin{lemma}\label{independentlemma}
Let $(M,\f,g_\intercal)$ be a Lorentzian foliation. Given a Riemannian metric $h$ on $M$ define, for each $k\in \mathbb{N}$, the bundle-like Lorentzian metric associated with $g_\intercal$ given by
$$g_k(X,Y) := g_\intercal(X,Y) +\frac{1}{k}h(X^\top,Y^\top), \quad \forall X,Y\in \mathfrak{X}(M),$$
where the superscript ``$\top$'' indicates the vertical part with respect to $h$ of the underlying vector/vector field. Then, for any leaf $L \in \f$ and any $p\in L$ we have
\begin{equation}\label{indlemmaeq}
I^\pm_\intercal(L;g_\intercal)=\bigcup_{k\in\mathbb{N}}I^\pm(p;g_k).
\end{equation}  
\end{lemma}
\begin{proof}
    We only give the proof for $I^-$ since the case of $I^+$ is entirely analogous. Given $q\in I^-_\intercal(p;g_\intercal)$, the Causal Waterfall Lemma (see \cite{us}, Lemma 4.19) implies there is a future-directed transversely $g_\intercal$-timelike curve $\gamma:[0,1]\rightarrow M$ such that $\gamma(0)=q, \gamma(1)=p$. Let $$ 
s:=\max\{h(\gamma^\prime(t)^\top, \gamma^\prime(t)^\top)\, : \,  t\in[0,1]\}.
    $$ It is clear from the compactness of $[0,1]$ that we can pick $K\in\mathbb{N}$ such that $k\geq K \implies |g_\intercal(\gamma^\prime(t),\gamma^\prime(t))|>\frac{1}{k}s, \forall t\in[0,1]$. It follows that $g_k(\gamma^\prime(t),\gamma^\prime(t))<0$ for $k\geq K$, which yields one inclusion. The converse inclusion is trivial.
\end{proof}

The next theorem is not yet a perfect generalization of the standard result of existence of time functions in stably causal spacetimes, since we have only been able to prove the continuity of the transverse time function ``apart from a set of measure zero''. We present it here in case future work - either by ourselves or other researchers - might sharpen it all the way. 

\begin{teo}\label{mainstableimpliestime}
If $(M,\f,g_\intercal)$ is transversely stably causal, then it admits a function $f:M\rightarrow\mathbb{R}$ which is almost everywhere continuous and strictly increases along every future-directed transversely causal curve.
\end{teo}
\begin{proof}
Since $(M,\f,g_\intercal)$ is transversely stably causal, we can fix a transverse Lorentzian metric $g_\intercal^\prime$ on $(M,\mathcal{F})$ with $g_\intercal<g^\prime_\intercal$ such that $(M,\f,g^\prime_\intercal)$ is transversely causal. By Theorem \ref{widen-the-cone-widen-the-wedge} and Corollary \ref{cortansvstableimpliesstable}, we can fix Lorentzian metrics $g \in \mathcal{BL}(g_\intercal), \, g' \in \mathcal{BL}(g'_\intercal)$, respectively, so that $g<g'$ and $(M,g')$ is a causal spacetime. We can write 
\begin{align*}
   g(X,Y) =\,  & g_\intercal(X,Y) + h(X^\top,Y^\top), \\
   g'(X,Y) = \, & g'_\intercal(X,Y) + \frac{f}{2}h(X^\top,Y^\top), \forall  X,Y \in \mathfrak{X}(M),
\end{align*}
where $h$ is a Riemannian metric on $M$, the superscript ``$\top$'' denotes the vertical part of the corresponding field according to $h$ and $f \in C^\infty(M)$ is chosen as in the proof of Theorem \ref{widen-the-cone-widen-the-wedge} to ensure that 
\begin{equation}\label{pegadinha}
    g'(v,v)\leq \frac{1}{2}g'_\intercal(v,v), \quad \forall v \in TM. 
\end{equation}

Next, define, for each $k \in \mathbb{N}$, 
\begin{align*}
   g_k(X,Y) :=\,  & g_\intercal(X,Y) + \frac{1}{k}h(X^\top,Y^\top), \\
   g'_k(X,Y) := \, & g'_\intercal(X,Y) + \frac{f}{2k}h(X^\top,Y^\top), \hspace{10pt} \forall  X,Y \in \mathfrak{X}(M).
\end{align*}
About this sequence of symmetric $(0,2)$-tensors on $M$, the following facts must be noticed. 
\begin{itemize}
    \item[1)] $g_k,g'_k$ are Lorentzian metrics on $M$, and in fact $g_k\in \mathcal{BL}(g_\intercal), g'_k\in \mathcal{BL}(g'_\intercal)$, so that in particular $(M,g_k)$ is stably causal by Corollary \ref{cortansvstableimpliesstable} and $(M,g'_k)$ is causal, $\forall k \in \mathbb{N}$. 
    \item[2)] For each $k \in \mathbb{N}$ we have $g_k<g'_k$. To see this, let $v\in TM$ be such that $g_k(v,v)\leq 0$. Then we directly check that $g_\intercal(v,v) \leq 0$, and hence $g'_\intercal(v,v)<0$ by our choice of $g'_\intercal$, which in turn ensures that $g'(v,v)<0$. However, 
    $$g'_k(v,v) = \left(1-\frac{1}{k}\right) g'_\intercal(v,v) + \frac{1}{k}g'(v,v) <0,$$
    whence the claim follows. 
\end{itemize}
Now, for each $k\in \mathbb{N}$, and each $\alpha\in[0,3]$, we define a symmetric $(0,2)$-tensor field on $M$ by
$$g^\alpha_k:=\frac{\alpha}{3} g_k^\prime+\left(1-\frac{\alpha}{3}\right)g_k.$$
Fix an admissible measure $\mu$ on $M$. According to the well-known discussion in \cite[section 6.4, pp. 199-201]{hawking-ellis}, for each $k \in \mathbb{N}$ the function 
$$f_k : p\in M \mapsto \int_1^2 v^\alpha_k(p)d\alpha \in \mathbb{R},$$ where $v^\alpha_k(p):=\mu(I^-(p;g_k^\alpha))$,
is a time function on $(M,g_k)$, and in particular it is continuous. We now set $v^\alpha_\intercal(p):=\mu(I^-_\intercal(p;g_\intercal^\alpha))$, where
$$
g_\intercal^\alpha:=\frac{\alpha}{3} g_\intercal^\prime+\left(1-\frac{\alpha}{3}\right)g_\intercal.
$$It is clear that each such $g_\intercal^\alpha$ is a transverse Lorentzian metric on $(M,\f)$ and for any $\alpha,\alpha'\in [0,3]$ such that $\alpha<\alpha'$ we have
$$g^\alpha_\intercal < g^{\alpha'}_\intercal.$$ Now, note that for each $k\in\mathbb{N}$, and each $\alpha\in[0,3]$ we have \begin{align*}
g_k^\alpha(X,X)&=\frac{\alpha}{3}g^\prime_k+\left(1-\frac{\alpha}{3}\right)g_k\\
&=\frac{\alpha}{3}\left(g^\prime_\intercal(X,X)+\frac{f}{2k}h(X^\top,X^\top)\right)+\left(1-\frac{\alpha}{3}\right)\left(g_\intercal(X,X)+\frac{1}{k}h(X^\top,X^\top)\right)\\
&=\frac{\alpha}{3} g_\intercal^\prime(X,X)+\left(1-\frac{\alpha}{3}\right)g_\intercal(X,X)+\frac{1}{k}\left(\frac{\alpha}{3}\frac{f}{2}+1-\frac{\alpha}{3}\right)h(X^\top, X^\top)\\
&=g_\intercal^\alpha(X,X)+\frac{1}{k}J(\alpha,f)h(X^\top, X^\top),
\end{align*} which means that each of them is a bundle-like metric associated to $g_\intercal^\alpha$. Therefore, we can employ Lemma (\ref{independentlemma}) to obtain that $v^\alpha_k$ is a monotone sequence converging pointwise to $v_\intercal^\alpha$ as $k\to+\infty$. Finally, we can set $f:M\rightarrow\mathbb{R}$ by $f(p):=\int_1^2v^\alpha_\intercal(p)d\alpha$ and by the Monotone Convergence Theorem, we have that $f_k$ converges almost everywhere uniformly (since $\mu(M)<+\infty$) to $f$, which means that $f$ is almost everywhere continuous. If $p, q\in M$ are on the same leaf, then for each $\alpha\in[0,3]$ we have $I^-_\intercal(p;g_\intercal^\alpha)=I^-_\intercal(q;g_\intercal^\alpha)\implies v^\alpha_\intercal(p)=v^\alpha_\intercal(q)\implies f(p)=f(q)$, which means $f$ is basic.

Finally, let $p<_\intercal q$ and denote respectively by $L$ and $L^\prime$ their respective leaves. Then, $L<_{M/\mathcal{F}}L^\prime$, and $I^-_{M/\mathcal{F}}(L)\subset I^-_{M/\mathcal{F}}(L^\prime)$, by the push-up lemma \cite[Thm. 4.23]{us}. Suppose we had $I^-_{M/\mathcal{F}}(L)= I^-_{M/\mathcal{F}}(L^\prime)$. By transverse stable causality of $g_\intercal$, we would have $L<<_{M/\mathcal{F}; g_\intercal^\prime}L^\prime$ with $g^\prime_\intercal$ transversely causal. Then, $L^\prime\in\overline{I^-_{M/\mathcal{F}}(L;g^\prime_\intercal)}$, which would imply that $L<<_{M/\mathcal{F}, g^\prime_\intercal}L^\prime$, which contradicts the transverse causality of $g^\prime_\intercal$. Then, $I^-_{M/\mathcal{F}}(L)\neq I^-_{M/\mathcal{F}}(L^\prime)$, and thus we can pick some leaf $L^{\prime\prime}\in I^-_{M/\mathcal{F}}(L^\prime) \setminus I^-_{M/\mathcal{F}}(L)$. Since $I^-_{M/\mathcal{F}}(L^\prime)$ is open, we can choose some leaf $L^*$ such that $L^{\prime\prime}<<_{M/\mathcal{F}}L^*<<_{M/\mathcal{F}}L ^\prime$. Now, if we had $L^*\in \overline{I^-_{M/\mathcal{F}}(L)}$, then we'd have $L^{\prime\prime}\in I^-_{M/\mathcal{F}}(L)$, which is a contradiction. Therefore, $I^-_{M/\mathcal{F}}(L^\prime)\setminus\overline{I^-_{M/\mathcal{F}}(L)}$ is a non empty open set, and consequently, so is $I^-_\intercal(q)\setminus\overline{I^-_\intercal(p)}^\intercal$ back in $M$. Then, the latter set has positive measure, which implies that $v^0_\intercal(p)=\mu(I^-_\intercal(p))<\mu(I^-_\intercal(q))=v^0_\intercal(q)$. The same reasoning can be employed to show that $v^\alpha_\intercal(p)<v^\alpha_\intercal(q), \forall \alpha\in[0,3]$ whenever $p<_{g^\alpha_\intercal}q$. We thus conclude that $f$ increases along future-directed transversely causal curves, since those volume functions are lower semi-continuous.
\end{proof}

\subsection{Transverse temporal functions}\label{subsec9.3}
Recall that a \textit{temporal function} on a spacetime $(M,g)$ is a smooth (i.e., $C^\infty$) function $f: M\rightarrow \mathbb{R}$ with past-directed timelike gradient $\nabla^gf$. It clear that any temporal function is also a time function. We wish to define the transverse analogues of temporal functions for the Lorentzian foliation $(M,\mathcal{F},g_\intercal)$, so that in particular these are smooth transverse time functions. Proposition \ref{proptransvtimevstime} suggests it is enough to consider ($\mathcal{F}$-)\textit{basic} smooth functions. 

\begin{definition}[Transverse temporal function]\label{defitemporalfunction}
   A \emph{transverse temporal function} on $(M,\mathcal{F},g_\intercal)$ is a smooth function $f:M\rightarrow \mathbb{R}$ such that 
   \begin{equation}\label{tempdefieq}
   df(v)>0 \text{ for any future-directed transversely causal } v\in TM.
   \end{equation}
\end{definition}
It is clear that any transverse temporal function is also a transverse time function. Indeed, we have the following smooth version of Proposition \ref{proptransvtimevstime}.
\begin{proposition}\label{proptranvtempvstemp}
    For a smooth function $f:M\rightarrow \mathbb{R}$ the following statements are equivalent. 
    \begin{itemize}
     \item[i)] $f$ is a transverse temporal function on $(M,\mathcal{F},g_\intercal)$.
     \item[ii)] $f$ is basic and for any bundle-like Lorentz metric $g$ on $(M,\mathcal{F})$ associated with $g_\intercal$, $f$ is a temporal function on $(M,g)$.
     \item[iii)] $f$ is basic and there exists a bundle-like Lorentz metric $g$ on $(M,\mathcal{F})$ associated with $g_\intercal$ such that $f$ is a temporal function on $(M,g)$. 
 \end{itemize}
 Moreover, if one (and hence all) of these statements holds, then for any bundle-like Lorentz metric $g$ on $(M,\mathcal{F})$ associated with $g_\intercal$, $\nabla^g f$ is a (past-directed, timelike) $g$-horizontal and projectable vector field. \end{proposition}
\begin{proof}
   $(i)\Rightarrow (ii)$\\
   Since we are assuming that $f$ is transverse temporal, it is also transverse time, and we already know by Prop. \ref{proptransvtimevstime} that it is basic. Given any any bundle-like Lorentz metric $g$ on $(M,\mathcal{F})$ associated with $g_\intercal$, we only need to show its is $g$-causal, because \eqref{tempdefieq} implies in that case it is past-directed $g$-timelike. The same equation evidently implies $\nabla^g f$ is everywhere nonzero. Suppose then we have $g_p(\nabla^gf(p),\nabla^gf) >0$ at some $p\in M$, i.e., $\nabla^g f(p)$ is a spacelike vector. The Lorentzian signature then implies that $\nabla^gf(p)^\perp \subset T_pM$ would be a Lorentz vector space and hence posses some timelike vector orthogonal to $\nabla^f(p)$ (see, e.g., the discussion between Lemmas 5.26 and 5.27 in \cite{oneillbook}), in violation of \eqref{tempdefieq}. We conclude that $f$ is a temporal function.\\
   $(ii)\Rightarrow (iii)$ is immediate, so we focus on $(iii)\Rightarrow (i)$. Thus, assume $f$ is basic and we have a bundle-like Lorentz metric $g$ on $(M,\mathcal{F})$ associated with $g_\intercal$ such that $\nabla^g f$ is everywhere past-directed $g$-timelike. Let $V\in \mathfrak{X}(\mathcal{F})$ be any vertical vector field. Since $f$ is basic we have
   $$0= Vf = g(\nabla^g f,V),$$
   whence we conclude that $\nabla^gf$ is $g$-horizontal. But then, given that any future-directed transversely causal vector $v\in TM$ has a $g$-causal horizontal part, we get 
   $$df(v)= g(\nabla^gf,v) = g(\nabla^gf, v_H) >0,$$
   where we have used that $\nabla^g f$ is past-directed $g$-timelike for the last inequality.

   All that remains to be seen is that $\nabla^gf$ is foliate. To that end, pick any vector field $X \in \mathfrak{X}(M)$ and any vertical vector field $V\in \mathfrak{X}(\mathcal{F})$. We now compute
   \begin{eqnarray}
    g([V,\nabla^g f ]_H,X) &=& g([V,\nabla^g f ]_H,X_H) = g_\intercal([V,\nabla^g f ],X) \nonumber \\
    &\stackrel{\mathcal{L}_Vg_\intercal =0}{=}& V(g_\intercal(\nabla^g f,X)) - g_\intercal(\nabla^gf, [V,X]) \nonumber \\
       &\stackrel{\text{$\nabla^gf$ is horizontal}}{=}& V(g(\nabla^g f,X)) - g(\nabla^gf, [V,X]) \nonumber \\
       &=& V(Xf) - [V,X]f \nonumber \\
       &=& X(Vf) \equiv 0, \nonumber
   \end{eqnarray}
   since $f$ is annihilated by $V$ for being basic. We conclude that $[V,\nabla^g f]_H\equiv 0$, and hence $[V,\nabla^g f]$ is indeed vertical, that is, $\nabla^g f$ is projectable as claimed.
\end{proof}

A natural question to be asked is whether the existence of a transverse time function implies the existence of a transverse temporal function. One could expect that the smoothing process employed in \cite{Bernal_2005} for the classical case would do. However, it is not clear for us how to carry out such a procedure without spoiling the property that the function is basic, so we defer this matter to future work. The main result of this section is the following.

\begin{teo}\label{tempfunctstable}
    If the Lorentzian foliation $(M,\mathcal{F},g_\intercal)$ admits a transverse temporal function, or equivalently, if there exists a bundle-like Lorentzian metric $g$ for $(M,\mathcal{F})$ associated with $g_\intercal$ such that $(M,g)$ admits a basic temporal function, then $(M,\mathcal{F},g_\intercal)$ is transversely stably causal.
\end{teo}

Before proving this theorem we need a few technical lemmas.

\begin{lemma}\label{templemma1}
Let $(M,\mathcal{F},g_\intercal)$ be a Lorentzian foliation, and let $\Omega >0$ be a smooth $\mathcal{F}$-basic function. If we define 
$$\hat{g}_\intercal := \Omega^2 g_\intercal,$$
then $(M,\mathcal{F},\hat{g}_\intercal)$ is a Lorentzian foliation, which is transversely stably causal if and only if $(M,\mathcal{F},g_\intercal)$ is. 
\end{lemma}
\begin{proof}
 For the first statement, just note that since $\Omega$ is basic, for any vertical vector field $V\in \mathfrak{X}(\mathcal{F})$ we have
 $$
 \mathcal{L}_V\hat{g}_\intercal = \cancelto{0}{2V(\Omega)} g_\intercal + \Omega ^2 \cancelto{0}{\mathcal{L}_Vg_\intercal} \equiv 0,
 $$
 and checking the other clauses of the definition of a transverse Lorentz metric is straightforward. 

 For the second statement, suppose $(M,\mathcal{F},g_\intercal)$ is transversely stably causal, and fix a point $p\in M$ and some nonvertical vector $v\in T_pM$. We can then pick a transversely causal Lorentzian metric $h_\intercal$ on $(M,\mathcal{F})$ such that 
 $$g_\intercal < h_\intercal.$$
 Thus,
 $$\hat{g}_\intercal(v,v)\leq 0 \Rightarrow g_\intercal(v,v) \leq 0 \Rightarrow h_\intercal(v,v) <0,$$
 whence we conclude that $\hat{g}_\intercal < h_\intercal$. For the converse, just work with the positive basic function $1/\Omega$. 
\end{proof}
\begin{lemma}\label{templemma2}
Let $(M,\mathcal{F},g_\intercal)$ be a Lorentzian foliation, and let $\Omega >0$ be a smooth $\mathcal{F}$-basic function. If $g$ is a bundle-like Lorentzian metric associated with $g_\intercal$, then $\hat{g}:= \Omega^2 g$ is a Lorentzian bundle-like metric associated with
$\hat{g}_\intercal := \Omega^2 g_\intercal.$ 
\end{lemma}
\begin{proof}
 First, observe that since $g$ and $\hat{g}$ are conformally related, $\forall p\in M$ we have 
 $$(T_pL_p)^{\perp_{g}}= (T_pL_p)^{\perp_{\hat{g}}},$$
 i.e., the horizontal distributions of $g,\hat{g}$ are the same. Therefore, given $\hat{g}$-horizontal projectable vector fields $X,Y\in \mathfrak{L}(\mathcal{F})$ we have, for any $V\in \mathfrak{X}(\mathcal{F})$,
 $$V(\hat{g}(X,Y)) = 2\cancelto{0}{V(\Omega)} g(X,Y) + \Omega ^2 \cancelto{0}{V(g(X,Y))} \equiv 0,$$
 where the first terms on the right-hand side in the previous equation vanishes because $\Omega$ is basic, while the second vanishes because $g$ is bundle-like and $X,Y$ are $g$-horizontal projectable fields. 

 We conclude that $\hat{g}$ is bundle-like, and it is clearly associated with $\hat{g}_\intercal$. 
\end{proof}

\begin{proof}[Proof of Theorem \ref{tempfunctstable}]
   Let $g$ be bundle-like Lorentzian metric associated with $g_\intercal$ and let $f\in C^\infty(M)$ be an $\mathcal{F}$-basic function such that $\nabla^g f$ is past-directed timelike on $(M,g)$. 

   By Prop. \ref{proptranvtempvstemp}, $\nabla^g f$ is a $g$-horizontal and projectable vector field. Therefore, the function 
   $$\Omega:= \sqrt{-g(\nabla^g f,\nabla^g f)}$$
   is a smooth positive $\mathcal{F}$-basic function. 
Let $\hat{g}:= \Omega^2 g$, $\hat{g}_\intercal := \Omega^2 g_\intercal$. By Lemma \ref{templemma1}, it is enough to show that $(M,\mathcal{F}, \hat{g}_\intercal)$ is stably causal, and by Lemma \ref{templemma2} $\hat{g}$ is a bundle-like Lorentzian metric associated with $\hat{g}_\intercal$. In addition, 
$$g(\nabla ^g,X) = Xf = \hat{g}(\nabla^{\hat{g}}f,X) = g(\Omega^2 \nabla^{\hat{g}}f,X), \quad X\in \mathfrak{X}(M),$$
whence 
$$\nabla^{\hat{g}}f =\Omega^{-2} \nabla ^gf $$
and thus
$$\hat{g}(\nabla^{\hat{g}}f,\nabla^{\hat{g}}f) = \Omega^{-2} g(\nabla^g f,\nabla^g f) \equiv -1.$$

To summarize, we conclude there is no loss of generality in just assuming $g(\nabla^g f,\nabla^g f)=-1$ by changing $g_\intercal,g$ by $\hat{g}_\intercal,\hat{g}$ if necessary, and we do for the rest of the proof, dropping altogether the hat symbol $\, \hat{}$. 

It is henceforth convenient to adopt, for any $X \in\mathfrak{X}(M)$, the decomposition
\begin{equation}\label{decompkerdf}
   X= -df(X)\nabla ^g f + X_{\ker df}, 
\end{equation}
where $X_{\ker df}$ denotes the $g$-orthogonal projection on the smooth codimension 1 distribution $\ker df\subset TM$. We also fix a number $\lambda >1$, and define a symmetric $(0,2)$-tensor field on $M$ by
\begin{equation}\label{lambdametric}
  g_\lambda(X,Y) := -\lambda df(X)df(Y) + g(X_{\ker df },Y_{\ker df}), \quad \forall X,Y \in  \mathfrak{X}(M). 
\end{equation}
We now proceed by a sequence of technical claims. \\

\noindent $\vdash$ \textit{Claim 1:} $g_\lambda$ is a Lorentz metric on $M$, and $\nabla^{g_\lambda}f$ is $g_\lambda$-timelike on $(M,g_\lambda)$.\\
Fix any $p \in M$, and let $E_2,\ldots,E_n$ denote some $g$-orthonormal frame of vector fields defined on an open set $U\ni p$ spanning the distribution $\ker df$ on $U$, Thus, if we put 
 
$$E_1:= \frac{\nabla^g f|_U}{\sqrt{\lambda}},$$ 
then we easily check that $g_\lambda(E_i,E_j)=\delta_{ij}$ for any $i,j\in \{2,\ldots, n\}$ and $g_\lambda(E_1,E_1) =-1$. In other words, $\{E_1, \ldots, E_n\}$ is a $g_\lambda$-orthonormal frame, from whose existence around any point of $M$ means that $g_\lambda$ is indeed a Lorentzian metric. In particular, on $U$ we have  
\begin{eqnarray*}
   \nabla^{g_\lambda}f &=& -g_\lambda(\nabla^{g_\lambda}f,E_1)E_1 +\sum_{i=2}^n \cancelto{0}{g_\lambda(\nabla^{g_\lambda}f,E_i)} \\
   &=& -\frac{df(\nabla^{g}f)}{\lambda}\nabla^{g}f \\
   &=& \frac{\nabla^{g}f}{\lambda} \equiv \frac{E_1}{\sqrt{\lambda}},
\end{eqnarray*}
which shows that $\nabla^{g_\lambda}f$ is indeed $g_\lambda$-timelike. \\

\noindent $\vdash$ \textit{Claim 2:} a $X\in \mathfrak{X}(M)$ is $g_\lambda$-horizontal if and only if $X$ is $g$-horizontal. \\
Indeed, given $X \in \mathfrak{X}(M)$ and $V \in \mathfrak{X}(\mathcal{F})$, $V$ is obviously a section of $\ker df$, and in particular $V=V_{\ker df}$, and thus,
$$g_\lambda(X,V) = g(X_{\ker df},V_{\ker df}) = g(X,V),$$
whence the Claim follows. $\dashv$ \\

\noindent $\vdash$ \textit{Claim 3:} $g_\lambda$ is bundle-like on $(M,\mathcal{F})$.\\
Let $X,Y \in \mathfrak{L}(\mathcal{F})$ be projectable $g$-horizontal vor fields (so by Claim 2 they are also $g_\lambda$ horizontal). Given any $V\in \mathfrak{X}(\mathcal{F})$ we have 
\begin{eqnarray*}
    [V,X_{\ker df}] &\stackrel{\eqref{decompkerdf}}{=}& [V,df(X)\nabla ^g +X] \\
    &=& V(g(\nabla^gf,X))\nabla^gf + df(X)[V,\nabla^g f] + [V,X],
\end{eqnarray*}
which is vertical because both $X$ and $\nabla ^gf$ are projectable (conf. Prop. \ref{proptranvtempvstemp}) and $g$ is bundle-like, so in particular 
$$V(g(\nabla^gf,X)) = V(df(X)) =0,$$
We conclude that $X_{\ker df} \in \mathfrak{L}(\mathcal{F})$, and analogously $Y_{\ker df }$ is also projectable. Therefore,
$$V(g_\lambda(X,Y)) = - \lambda \cancelto{0}{V(df(X))}df(Y) -\lambda df(X)\cancelto{0}{V(df(Y))} + V(g(X_{\ker df},Y_{\ker df})) \equiv 0,$$
whence the proof of the Claim is complete. $\dashv$\\

Given Claims 1-3, it follows that the symmetric $(0,2)$-tensor field $g^\lambda _\intercal$ on $M$ defined by 
$$g^\lambda _\intercal(X,Y) := g_\lambda(X_H,Y_H), \quad \forall X,Y \in \mathfrak{X}(M),$$
where the subscript $H$ indicates both $g_\lambda$ and $g$-horizontality as per Claim 2, defines a transversely causal transverse Lorentzian metric on $(M,\mathcal{F})$. 

Finally, let $v\in TM$ be a $g$-transversely causal vector, so in particular $g(v_H,v_H) \leq 0$. We now compute 
\begin{eqnarray*}
    g^\lambda_\intercal(v,v) &=& - \lambda (df(v_H))^2 + g((v_H)_{\ker df},(v_H)_{\ker df}) \\
    &<& - (df(v_H))^2 + g((v_H)_{\ker df},(v_H)_{\ker df}) \\
    &\equiv & g(v_H,v_H) \leq 0,
\end{eqnarray*}
so we conclude that 
$$g_\intercal < g^\lambda _\intercal,$$
and and thus that $(M,\mathcal{F},g_\intercal)$ is indeed transversely stably causal. 
\end{proof}

\begin{exe}[Spatially homogeneous foliations]\label{homogeneousexe}
    {\em 
    
Let $(M,g)$ be a spacetime. Assume a Lie group $G$ has an isometric (say left) smooth action $\theta$ on $M$ so that all the orbits of the action are (maybe only immersed) submanifolds of fixed codimension. Thus, these orbits comprise a so-called \emph{homogeneous foliation} $\f$ of $M$. We further assume that each orbit is a \textit{spacelike} (again, maybe only immersed) submanifold of codimension $\geq2$, and refer to the foliation as \textit{spatially homogeneous} in this case. The fact that $\theta$ is isometric with spacelike orbits implies that the metric $g$ is bundle-like for $(M,\f)$, which has thus a natural structure of a transversely time-oriented Lorentzian foliation. Furthermore, if in addition $(M,g)$ admits a smooth temporal function $f:M\rightarrow \mathbb{R}$ which is invariant by the action $\theta$ (which is not always the case - see Remark \ref{quitecutermk}), then  by Proposition \ref{proptranvtempvstemp} $f$ is a transverse temporal function for that Lorentzian foliation, which is therefore transversely stably causal by Theorem \ref{tempfunctstable}. 
    
    }
    
\end{exe}

\begin{exe}[Foliations by suspensions]\label{suspensionexe}
    {\em 
    
Let $(S^n,h)$ be a stably causal spacetime, and fix a smooth temporal function $f:S\rightarrow \mathbb{R}$ thereon. Let $B^k$ be some smooth connected manifold and assume there is a smooth (say left) action $\alpha$ of $\pi_1(B)$ on $(S,h)$ by isometries leaving $f$ invariant. Denote by $\rho:\hat{B}\rightarrow B$ the universal covering map of $B$. Fix any Riemannian  metric $\eta$ on $B$, so that the pullback metric $\hat{\eta}:= \rho^\ast \eta$ on $\hat{B}$ is also Riemannian. The metric $\hat{g}=\hat{\eta }\oplus h$ on $\hat{M}:= \hat{B} \times S$ is naturally a time-oriented Lorentzian metric, and the smooth function $\hat{f}:\hat{M}\rightarrow \mathbb{R}$ given by 
$$\hat{f}(\hat{b},x) := f(x), \quad \forall \hat{b} \in \hat{B}, \forall x\in S$$
 is a temporal function on $(\hat{M}, \hat{g})$, which is therefore also a  stably causal spacetime. Consider the foliation $\hat{\f}$ given by the fibers of the submersion $proj_2:(\hat{b},x) \in \hat{M}\mapsto x \in S$. It is clear that $\hat{g}$ is bundle-like with respect to $\hat{\f}$ and $\hat{f}$ is $\hat{\f}$-basic, so this foliation is a transversely stably causal foliation. Now, there is a smooth, free, proper right action $\hat{\alpha}$ of $\pi_1(B)$ on $\hat{M}$ given by 
 $$(\hat{b},x) \in \hat{M}\mapsto (\hat{b}\cdot [\gamma],\alpha_{[\gamma]^{-1}}(x)) \in \hat{M}, \quad \forall [\gamma] \in \pi_1(B), $$
 where in the first slot we have used the canonical right action of $\pi_1(B)$ on $\hat{B}$ by deck transformations, so that $\hat{\alpha}$ is $\hat{g}$-isometric and takes leaves of $\hat{\f}$ into leaves. We conclude that the quotient $M:= \hat{M}/\pi_1(B)$ has the structure of a smooth manifold and inherits an induced foliation $\f$ and a Lorentz metric $g$ which remains bundle-like with respect to $\f$. Since $\hat{f}$ is invariant by $\hat{\alpha}$ it induces a smooth $\f$-basic temporal function $\overline{f}$ on $(M,g)$. By Proposition \ref{proptranvtempvstemp} it is a transverse temporal function for the associated Lorentzian foliation, which is therefore transversely stably causal by Theorem \ref{tempfunctstable}. $\f$ is said to be given by a \textit{suspension} of the action $\alpha$. (A simple particular case is obtained by fixing any isometry $\phi:S\rightarrow S$, taking $B=\mathbb{S}^1$, so that  $\pi_1(B) = \mathbb{Z}$, and defining $\alpha_m(x) := \phi^m(x), \forall x \in S, \forall m \in \mathbb{Z}$.)
    }
    
\end{exe}

\subsection{Spatial symmetries in standard stationary spacetimes}\label{stationarysubsec}

 The class of spatially homogeneous foliations in a spacetime $(M^{n+1},g)$ discussed in the Example \ref{homogeneousexe} above, to be relevant, requires the existence of a temporal function which is invariant by a Lie group of isometries with spacelike orbits. One might wonder if such functions exist in a sufficiently broad, interesting class of known spacetimes. We show in this subsection that they do in the class of \textit{standard stationary spacetimes} provided these have added, ``spatial'' symmetries. 

More precisely, we assume that there exist a connected Riemannian manifold $(M_0^n,g_0)$, a 1-form $\omega_0 \in \Omega^1(M_0)$ and a smooth positive function $\beta_0 \in C^\infty(M_0)$ such that $M=\mathbb{R}\times M_0$ and 
\begin{equation}\label{standardstationaryeq}g=  -\beta_0^2d pr_1\otimes d pr_1 +pr_2^\ast\omega_0\otimes d pr_1 + d pr_1 \otimes pr_2^\ast \omega_0 + pr_2^\ast g_0,\end{equation}
where $pr_1: \mathbb{R}\times M_0 \rightarrow \mathbb{R}$ and $pr_2 : \mathbb{R}\times M_0 \rightarrow M_0$ are the standard projections onto the first and second factor, respectively. (By a classic result of Javaloyes and Sánchez \cite{javasanchez} a spacetime $(M,g)$ is isometric to such a product with metric as in \eqref{standardstationaryeq} if it is distinguishing and admits a complete timelike Killing vector field.)

We assume in addition that $(M_0,g_0)$ carries a smooth isometric left action $\alpha$ of a Lie group $G$ which preserves $\beta$ and $\omega_0$
$$\alpha_\eta^\ast \omega_0 = \omega_0 \text{ and }\beta_0\circ \alpha_\eta = \beta_0, \quad \forall \eta \in G.$$
We assume, in addition, the orbits to have a fixed codimension $\geq 1$ in $M_0$. Then the action obviously induces an isometric action $\hat{\alpha}$ in $(M,g)$ with spacelike orbits by 
$$\hat{\alpha}_\eta (t,x) := \alpha_\eta(x), \quad \forall \eta \in G, \forall t \in \mathbb{R}, \forall x\in M_0. $$
Therefore, we get a spatially homogeneous foliation $\f$ in $M$ as in Example \ref{homogeneousexe} for which $g$ is bundle-like. We claim: 
\begin{proposition}\label{stationaryexampleprop}
 The function $f:= pr_1$ is a temporal $\f$-basic function, so $\f$ is a transversely stably causal foliation. 
\end{proposition}
\begin{proof}
  Since $f$ is constant on the orbits of $\hat{\alpha}$ by construction, all that remains to be shown is that $\nabla ^gf$ is indeed timelike. But this is clear, because if we introduce local coordinates $(x^1, \ldots,x^n)$ on $M_0$, so that together with $f=:x^0$ they define a local coordinate system for $M$, we have, via the $(00)$-cofactor, 
  $$g(\nabla^g f,\nabla^g f) \equiv  g^{00} = \frac{\det [(g_0)_{ij}]}{\det [g_{ab}]} <0,$$
  since $g_0$ is positive-definite while $g$ has Lorentz signature and hence a negative determinant. 
  
\end{proof}
Finally, we observe that the assumptions in this subsection do hold, e.g., in the class of \textit{stationary and axisymmetric} models describing the all-important Kerr-Newman family of black holes in General Relativity (see \cite{wald_general_1984} for more details). 

\section{Transverse $K$-causality}\label{sec:k-causal}

Among the different equivalent formulations of stable causality, so far we have dealt with the ones with a more evident geometric aspect (items 1-4 in Figure (\ref{equivalences})). Nonetheless, there exists an extensive literature (see \textit{e.g.} \cite{Minguzzi_2007a}, \cite{Minguzzi_2007}, \cite{Minguzzi_2009}) on a relation-theoretic (items 5-6 in the Introduction) approach to stable causality. We give a number of partial transverse analogs here.

Fix a transversely time-oriented Lorentzian foliation $(M,\f,g_\intercal)$. The analogues of Seifert's and $K^+$ relations on $M/\f$ is straightforward. Indeed, $$J^+_{S\intercal}:=\bigcap_{g_\intercal^\prime>g_\intercal}J^+_{M/\mathcal{F},g^\prime_\intercal}$$ is called the \textit{transverse Seifert relation}, whereas the \textit{transverse $K^+$ relation} on $(M, \mathcal{F}, g_\intercal)$ is defined as the smallest closed and transitive relation on $M/\mathcal{F}$ which contains the transverse causal relation $J^+_{M/\mathcal{F}}$. 

The contrast with the usual theory arises when one tries to define a transverse analogue of $K$-causality. In fact, we shall \textit{not} define $(M,\mathcal{F},g_\intercal)$ to be transversely $K$-causal when the $K^+_\intercal$ relation is antisymmetric. Indeed, \textit{doing so would immediately make it false that transverse stable causality implies $K$-causality}, as the next simple example shows.

\begin{exe}\label{exem-antisym-fails}
    {\em Consider the foliation $\f$ on $M:=\mathbb{R}^3\setminus\{(0,0,0)\}$ given by the connected components of the fibers of the submersion $(t,x,y)\mapsto(t,x)$, and let $z_1=(0,0,-1), z_2=(0,0,1)$. If we endow the foliation $\f$ with the transverse Lorentzian metric $g_\intercal=-dt^2+dx^2$, it is clear that the flat Minkowski metric $\eta$ restricted to $M$ is bundle-like, and $f\in C^\infty(M)$ given by $f(t,x,y)=t$ is $g$-temporal and $\f$-basic. Hence $(M,\f,g_\intercal)$ is transversely stably causal. However, as Lemma \ref{lookatthat} below will show, both $(\hat{z}_1,\hat{z}_2)$ and $(\hat{z}_2,\hat{z}_1)$ are contained in $K^+_\intercal$, where $\hat{z}_i$ is the leaf containing $z_i$. 
    }
\end{exe}

\medskip

Note that saying that a relation $R\subset X\times X$ is antisymmetric means that $$\{(x,y),(y,x)\}\subset R\implies(x,y)\in \Delta_X,$$ where $\Delta_X$ is the diagonal of $X\times X$. To deal with the issue at hand, it is clear that we must switch $\Delta_{M/\f}$ to a larger relation on $M/\f$. Now, recall that a topological space $(X,\tau)$ is Hausdorff if and only if $\Delta_X$ is closed in $X\times X$ (with the product topology). Its closure $\overline{\Delta_X}$ in $X\times X$ is a reflexive and symmetric, but not necessarily transitive relation. Also, if $\tau_z:=\{\mathcal{U}\in\tau: z\in\mathcal{U}\}$, it is easy to check that we can write
$$
\overline{\Delta_X}=\{(x, y)\in X\times X:\forall \mathcal{U}\in \tau_x, \forall \mathcal{V}\in \tau_y,\ \mathcal{U}\cap\mathcal{V}\neq\emptyset\}.
$$Note that the property on the right-hand side of the previous equality is possessed by $\hat{z}_1$ and $\hat{z}_2$ in Example \ref{exem-antisym-fails}. Hence $(\hat{z}_1,\hat{z}_2),(\hat{z}_2,\hat{z}_1) \in \overline{\Delta_{M/\f}}$ there. 

Hence, a possible suitable definition for \textit{transversely K-causality} is requiring that whenever $(\hat{x},\hat{y})$ and $(\hat{y},\hat{x})$ are in $K^+_\intercal$ we have that $(\hat{x},\hat{y})\in \overline{\Delta_{M/\f}}$. 

We also say that $(M,\f,g_\intercal)$ satisfies the \textit{transverse Seifert condition} if $(\hat{x},\hat{y})\in J^+_{S\intercal}$ and $(\hat{y},\hat{x})\in J^+_{S\intercal}\implies(\hat{x},\hat{y})\in \overline{\Delta_{M/\f}}$. 

\begin{proposition}\label{stablemeansseifert}
The following implications apply.
    \begin{itemize}
    \item[i)] If the transverse Seifert condition holds on $(M,\f,g_\intercal)$ and the transverse Seifert relation is closed, then $(M,\f,g_\intercal)$ is transversely $K$-causal. 
        \item[ii)] If $(M,\f,g_\intercal)$ is transversely stably causal, then the transverse Seifert relation is antisymmetric. In particular, the transverse Seifert condition holds. 
    \end{itemize} 
\end{proposition}
\begin{proof}
    $(i)$\\
    Since The Seifert relation is clearly transitive and contains $J^+_{M/\f, g_\intercal}$, if it is in addition closed, then it contains $K^+_\intercal$ by the definition of the latter. \\
    $(ii)$\\
    Let $(\hat{x},\hat{y}),(\hat{y},\hat{x})\in J^+_{S\intercal}$. By definition of transverse stable causality we can pick a transverse Lorentz metric $g'_\intercal >g_\intercal$ such that $(M,\f,g'_\intercal)$ is transversely causal. Since $$J^+_{S\intercal}\subset J^+_{M/\f,g'_\intercal},$$
    and the latter relation antisymmetric due to transverse causality, we conclude that $\hat{x}=\hat{y}$, thus proving that the transverse Seifert relation is indeed antisymmetric.
\end{proof}
\begin{corollary}\label{massamassa}
 If $(M,\f,g_\intercal)$ is transversely stably causal and the transverse Seifert relation is closed, then $(M,\f,g_\intercal)$ is transversely $K$-causal.   
\end{corollary}
\qcd

\begin{rmk}
    {\em It is well-known that in spacetimes the Seifert relation is always closed. However, the extant proofs (see, e.g. \cite[Thm. 4.94]{minguzzicausality}) rely very heavily on the existence of causal limit curve theorems and normal convex neighborhoods, and it is not obvious to us whether, or how, these could be either adapted or bypassed in the present context. However, if we impose a kind of ``local causal regularity'' by hand, then the desired closure of the transverse Seifert relation holds. This motivates the next definition.}
\end{rmk}

\begin{definition}[Causal leaf-regularity]\label{regulardefi}
  Given the Lorentzian foliation $(M,\f, g_\intercal)$, we say that a leaf $L\in \f$  is \emph{causally regular} if there exists a saturated neighborhood $\mathcal{B}\supset L$ of $L$ in $M$ with the following property: Given nets $(L'_i)_{i\in I}$ and $(L''_i)_{i\in I}$ of leaves contained in $\mathcal{B}$ such that $(i)$ $L'_i \stackrel{M/\f}{\longrightarrow} L' \subset \mathcal{B}$ and $L''_i \stackrel{M/\f}{\longrightarrow} L'' \subset \mathcal{B}$, and $(ii)$ for each $i\in I$, either $L'_i=L''_i$ or there exists a future-directed transversely causal curve segment starting at $L'_i$, ending at $L''_i$ and entirely contained in $\mathcal{B}$, then either $L'=L''$ or else there exists a future-directed transversely causal curve segment in $\mathcal{B}$ from $L'$ to $L''$. We say that $(M,\f, g_\intercal)$ is \emph{causally leaf-regular} if each $L\in \f$ is causally regular.
\end{definition}

Although causal leaf-regularity property may seem an \textit{ad hoc} property at first, we shall see below (cf. Proposition \ref{simplesubexe}) that in foliations given by connected components of certain submersions, convex normal neighborhoods of the base can be pulled back to $M$ to yield regular neighborhoods, so that property does hold in these cases. 

The usefulness of the notion of causal leaf-regularity now becomes apparent:
\begin{teo}\label{theone}
    If $(M,\f, g_\intercal)$ is causally leaf-regular and its leaf space $M/\f$ is Hausdorff, then its transverse Seifert relation is closed. In particular, in this case, if $(M,\f,g_\intercal)$ is transversely stably causal, then it is transversely $K$-causal.  
\end{teo}
\begin{proof}
    Let $(L,L') \in \overline{J^+_{S\intercal}}\setminus \Delta_{M/\f}$. Fix causally regular neighborhoods $\mathcal{B}\supset L$ and $\mathcal{B}'\supset L'$. By the Hausdorffness of the leaf space, we can assume without loss of generality that $\mathcal{B}\cap \mathcal{B}'=\emptyset$. Then, there exist nets $(L_i)_{i\in I}$ and $(L'_i)_{i\in I}$ of leaves such that $(i)$ $L_i \stackrel{M/\f}{\longrightarrow} L \in \mathcal{B}$ and $L'_i \stackrel{M/\f}{\longrightarrow} L' \in \mathcal{B}'$, and $(ii)$ for each $i\in I$ there exists a future-directed $g_\intercal$-transversely causal curve segment $\gamma_i:[0,1] \rightarrow M$ starting at $L_i$ and ending at $L'_i$. Fix $x\in L$ and $y \in L'$. Consider relatively compact open sets $U\ni x$ and $V\ni y$ in $M$ so that $\overline{U}\subset \mathcal{B}$ and $\overline{V}\subset \mathcal{B}'$. By virtue of the \textit{Causal Waterfall Lemma} (see \cite{us}, Lemma 4.19) we can assume that $\gamma_i(0) \rightarrow x$, $\gamma_i(1) \rightarrow y$ so we eventually have $0<t_i<s_i<1$, $\gamma_i(t_i)\in \partial U$, $\gamma(s_i) \in \partial V$, $\gamma_i[0,t_i) \subset U$, $\gamma_i(s_i,1] \subset V$. By compactness of $\partial U,\partial V$ we must have, up to passing to a subnet, that $\gamma_i(t_i) \rightarrow c \in \partial U$ and $\gamma_i(s_i)\rightarrow c'\in \partial V$. But then, causal regularity gives $(L,L_c),(L_{c'},L') \in J^+_{M/\f,g_\intercal}$, and $(L_c,L_{c'})\in  \overline{J^+_{M/\f, g_\intercal}}$. Given any transverse Lorentz metric $g'_\intercal >g_\intercal$ we thus have $(L,L_c),(L_{c'},L') \in I^+_{M/\f,g'_\intercal}$, and $(L_c,L_{c'})\in  \overline{J^+_{M/\f, g'_\intercal}}$, and thus $(L,L') \in I^+_{M/\f, g'_{\intercal}}$ by the openness of the transverse chronology (see \cite{us}, Corollary 4.27). Since $g'_\intercal$ has been taken arbitrarily, we conclude that $(L,L') \in J^+_{S\intercal}$, and the proof is complete.
\end{proof}
When the leaf space $M/\mathcal{F}$ is Hausdorff, then $\Delta_{M/\f}$ is closed in $M/\f\times M/\f$, whence $\overline{\Delta_{M/\f}}=\Delta_{M/\f}$ and thus transverse $K$-causality and the transverse Seifert condition reduce to usual anti-symmetry requirement on the  corresponding relation in this case.

An investigation by Minguzzi in \cite{mingutility} brought from mathematical economy the notion of an \textit{utility function} on a set with a preorder $(X, R)$ (that is, $R$ is reflexive and transitive), by which one means a function $f:X\rightarrow\mathbb{R}$ such that $(i)$ $(x,y)\in R\implies f(x)\leq f(y)$, and $(ii)$ if $(x,y) \in R$ but $(y,x)\notin R$, then $f(x)<f(y)$. In \cite{levin} the author discusses sufficient conditions for the existence of \textit{continuous} utility functions.

\begin{teo}[Levin]
If $X$ is a second countable, locally compact Hausdorff topological space, and $R$ is a closed preorder on $X$, then there exists a continuous utility function on $(X, R)$. Moreover, denoting with $\mathfrak{U}$ the set of continuous utilities we have that the
preorder $R$ can be recovered from the continuous utility functions, namely  $$
(x,y)\in R\iff \forall u\in\mathfrak{U}, u(x)\leq u(y).
$$
\end{teo}

Observe that the leaf space is always second countable and locally compact, but usually not Hausdorff. If it Hausdorff, however, then Levin's theorem has the immediate consequence: 
\begin{teo}\label{theoKimpliestimefunc}
    If $M/\mathcal{F}$ is Hausdorff, then transverse $K$-causality implies the existence of a transverse time function.
\end{teo}
\begin{proof}
By Levin's theorem, there exists a continuous utility function $t:M/\mathcal{F}\rightarrow \mathbb{R}$ for the transverse $K^+-$relation, that is,  $(\hat{x},\hat{y})\in K^+_\intercal \implies t(\hat{x})\leq t(\hat{y})$ and $(\hat{x},\hat{y})\in K^+_\intercal$ and $(\hat{y},\hat{x})\notin K^+_\intercal)\implies t(\hat{x})<t(\hat{y})$. Then, if $\hat{x}$ and $\hat{y}$ are distinct points in $M/\mathcal{F}$ with $\hat{x}<_{M/\mathcal{F}}\hat{y}$, then $(\hat{x},\hat{y})\in K^+_\intercal\setminus \Delta$, where $\Delta\subset M/\mathcal{F}\times M/\mathcal{F}$ is the diagonal. Since the $K^+_\intercal-$relation is antisymmetric it follows that $(\hat{y},\hat{x})\notin K^+_\intercal$, which implies that $t(\hat{x})<t(\hat{y})$. Thus, given the projection $\pi_{\f}:M\rightarrow M/\f$,  $f:=t\circ\pi_\mathcal{F}:M\rightarrow\mathbb{R}$ is a transverse time function.
\end{proof}
This result together with Theorem \ref{theone} yields
\begin{corollary}\label{tehonecor}
    If $(M,\f, g_\intercal)$ is transversely stably causal, causally leaf-regular and its leaf space $M/\f$ is Hausdorff, then it admits a transverse time function. 
\end{corollary}
\qcd
An important class of examples where Hausdorff leaf spaces appear is that involving \textit{Lorentzian orbifolds} (see \cite{us} for a more detailed discussion of these objects). As is well known (see, \textit{e.g.}, {\cite[Proposition 2.23]{mrcun}}), every orbifold - which is Hausdorff by definition - can be seen as the leaf space of a regular homogeneous foliation with compact leaves. So if an orbifold $\mathcal{O}\cong M/\f$ is endowed with a Lorentzian orbifold metric $g_\mathcal{O}$, one can study its geometry via the transverse geometry of $(M, \f, g_\intercal(=\pi_\mathcal{F}^*g_\mathcal{O}))$, where $\pi_\mathcal{F}: M\rightarrow\mathcal{O}$. It is straightforward to check (as is discussed in the final remarks of \cite{us}) that the transverse stable causality of $(M, \f, g_\intercal(=\pi_\mathcal{F}^*g_\mathcal{O}))$ is the same as ``opening up'' the cones on $\mathcal{O}$ arising from $g_\mathcal{O}$. Therefore, the meaning of a Lorentzian orbifold being time-oriented and stably causal is clear. What is not clear to us, however, is whether that implies $K$-causality, say, via the closure of the analogue of the Seifert relation, and in particular whether causal leaf-regularity holds in this case (we suspect it does).  Thus we again defer an appropriate discussion of these issues to future work.

\section{Transverse stable causality in simple foliations}\label{sec:submersions}

Taking the equivalences laid out in Figure \ref{equivalences} as a plan for our generalization work, so far we have been able to establish some corresponding implications for a general foliation, as is shown in Figure \ref{tequivalences}.

\begin{figure}[h]
\centering{
\scalebox{0.8}{
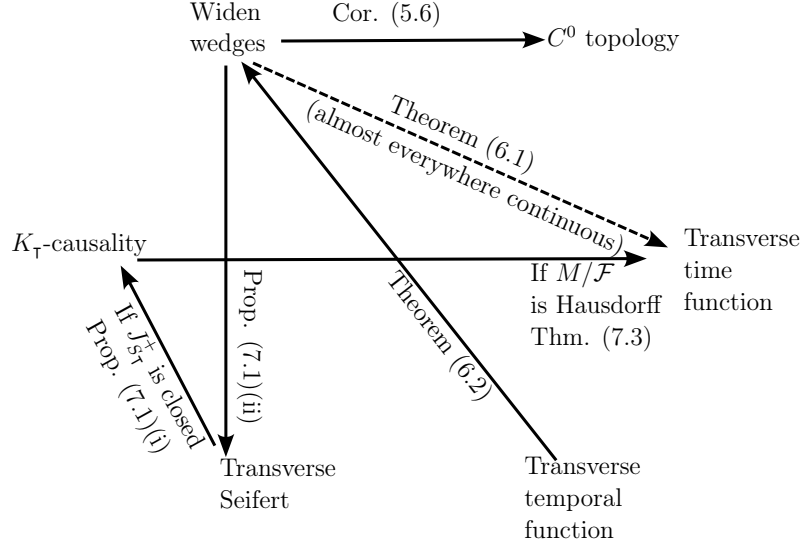}
\caption{Diagram of the logical implications established so far for a general foliation.}
\label{tequivalences}}
\end{figure}

In this section we deal with the special case where $\mathcal{F}$ is given by an onto submersion $\pi:M\rightarrow B$. If the fibers of $\pi$ are connected, we have simply that $M/\mathcal{F}\cong B$ and the transverse causal structure of $(M,\mathcal{F})$ is identical with the causal structure of $B$. However, if the fibers of $\pi$ are not connected, then the leaves of $\mathcal{F}$ are taken as the connected components of them, and the leaf space is a $T_1$ but not $T_2$ manifold (see \cite[Prop. 2.8]{us}), and there is a natural continuous surjection $\hat{\pi}:M/\f\rightarrow B$ such that $\hat{\pi}\circ\pi_\f=\pi$. In either case a transverse Lorentzian metric $g_\intercal$ on $(M,\mathcal{F})$ can be easily projected to yield a Lorentzian metric $h$ on $B$. We also implicitly assume that compatible time-orientations were chosen for $(M,\mathcal{F}, g_\intercal)$ and $(B,h)$.

\begin{proposition}\label{propussub}
    Let $\pi:M\rightarrow B$ be an onto submersion, and consider the foliation $\f$ whose leaves are the connected components of the fibers. Let $h$ be a Lorentzian metric on $B$ so that $(B,h)$ is time-oriented, and consider the Lorentzian foliation $(M,\f, g_\intercal:= \pi^\ast h)$ with the induced time-orientation. Then, $\forall x, y\in M$ we have that
 $$\f_x\ll_{M/\f} \f_y \text{ [resp. $\f_x<_{M/\f} \f_y$]} \implies \pi(x)\ll_h \pi(y) \text{ [resp. $\pi(x)<_h \pi(y)$]}.$$   
    
    If furthermore it holds that for all $b\in B$, there is an open neighborhood $\mathcal{U}\ni b$ such that every fiber of $\pi$ contained in $\pi^{-1}(\mathcal{U})\setminus\pi^{-1}(b)$ is connected, then the converse holds.
\end{proposition}
\begin{proof}
($\implies$) Trivial.\newline
($\impliedby$) Start by fixing a future-directed timelike curve $\alpha:[0,1]\rightarrow B$ with $\alpha(0)=\pi(x)$ and $\alpha(1)=\pi(y)$. Our goal is to construct a suitable lift of $\alpha$ starting at $x$ and ending at $y$. To that end, begin by covering $\alpha([0,1])\subset B$ by a finite collection $\mathcal{U}_1,\ldots, \mathcal{U}_k$ of open sets of $B$ which are domains of local sections $\sigma_i: \mathcal{U}_i\subset B \rightarrow M$ of $\pi$. We can choose $\sigma_1$ to satisfy $\sigma_1(\pi(x))=x$. Now, let $\{t_0=0< t_1< \cdots, t_k=1\}$ be a partition of $[0,1]$ with $\alpha([t_{i-1},t_i])\subset \mathcal{U}_i$ for each $i\in \{1,\ldots, k\}$. For each such $i$ one can define $$\widehat{\alpha}_i: t\in [t_{i-1},t_i] \mapsto (\sigma_i\circ \alpha) (t) \in M.$$
We easily check that $\widehat{\alpha}_i$ is a lift of $\alpha|_{[t_{i-1},t_i]}$ through $\pi$ and that it is a transversely timelike curve segment with respect to $(M,\f, g_\intercal)$. In particular, $\widehat{\alpha}_{1}(0)=x$, $\widehat{\alpha}_{k}(1)\in \pi^{-1}(\pi(y))$ and $\widehat{\alpha}_{i}(t_i)$ and $\widehat{\alpha}_{i+1}(t_i)$ will be on the same fiber of $\pi$. However, since not all fibers are connected, we cannot \textit{a priori} apply the Waterfall construction to connect the curves. To circumvent this issue, suppose that for some $i$ we have that $\pi^{-1}(\alpha(t_i))$ is not connected. By hypothesis, there exists an open neighborhood $\alpha(t_i)\subset\mathcal{U}\subset B$ such that all the fibers inside $\pi^{-1}(\mathcal{U})$, except for $\pi^{-1}(\alpha(t_i))$ are connected. Then, we can tilt $t_i$ slightly to $t_i+\varepsilon$ such that $\alpha(t_i+\varepsilon)$ is still in $\mathcal{U}_{i}\cap\mathcal{U}_{i+1}$ and $\pi^{-1}(\alpha(t_i+\varepsilon))$ is connected. Doing this as many times as needed, we ensure that the Waterfall construction can be applied $k-1$ times and we obtain a single piecewise smooth future-directed transversely timelike curve $\widehat{\alpha}:[0,1]\rightarrow M$ with $\widehat{\alpha}(0)=x$ and $\widehat{\beta}(1)=y$ as desired; we conclude that $L^x\ll_{M/\f} L^y$. The causal case is entirely analogous.
\end{proof}

The previous result naturally motivates the following definition. 
\begin{definition}[Simple foliations]\label{simpledefi}
   A Lorentzian foliation $(M,\f,g_\intercal)$ is said to be \emph{simple} if there exist a spacetime $(B,h)$ and 
   an onto submersion $\pi:M\rightarrow B$ be an onto submersion so that 
   \begin{itemize}
       \item the leaves of the foliation $\f$ are the connected components of the fibers of $\pi$;
       \item $g_\intercal = \pi^\ast h$ and the transverse time-orientation on $(M,\f,g_\intercal)$ is the one induced by the time-orientation on $(B,h)$;
       \item for all $b\in B$, there is an open neighborhood $\mathcal{U}\ni b$ such that every fiber of $\pi$ contained in $\pi^{-1}(\mathcal{U})\setminus\pi^{-1}(b)$ is connected.
       \end{itemize}
\end{definition}

\begin{exe}
{\em
Consider the submersion $\pi:\mathbb{R}^3\setminus\mathbb{Z}^3\rightarrow \mathbb{R}^2$ given by $\pi(t,x,y):= (t,x)$. Note that whenever $(t,x)\in\mathbb{Z}^2$ we have that $\pi^{-1}(t,x)$ has a countably infinite number of connected components. However, it is easily seen that the foliation $\f$ given by the connected components of the fibers of $\pi$ satisfies the condition of Definition \ref{simpledefi}, by taking, for example, $\mathcal{U}$ as the open ball of radius 1 centered on $(t,x)$, and choosing $h = -dt^2 + dx^2$ on the codomain and $g_\intercal=-dt^2 + dx^2$ on the domain. Thus, $(M,\f,g_\intercal)$ is indeed simple, and we can verify that transverse stable causality holds thereon. The leaf space of $\f$ satisfies almost all the conditions making it a smooth $2$-manifold, except for its topology: it is second countable and locally homemorphic to $\mathbb{R}^2$, but it is $T_1$ instead of Hausdorff, with each point with integer coordinates having a countably infinite number of branched copies, similar to the well-known example of the ``double-headed snake'' \cite[p. 43]{tti}.
}
\end{exe}

We shall for the remainder of this section, fix a simple foliation $(M,\f,g_\intercal)$, a spacetime $(B,h)$ and the submersion $\pi:M\rightarrow B$ satisfying the conditions in Definition \ref{simpledefi}. 

\begin{proposition}\label{simplesubexe}
   Every simple Lorentzian foliation is causally leaf-regular.
\end{proposition}
\begin{proof}
  As indicated above, we fix a spacetime $(B,h)$ and the submersion $\pi:M\rightarrow B$ satisfying the conditions in Definition \ref{simpledefi} associated with the simple foliation $(M,\f,g_\intercal)$. Let $L \subset \pi^{-1}(b)$ be a leaf, with $b\in B$. Let $b\in \mathcal{N}\subset B$ be a convex normal neighborhood in $(B,h)$, and $\mathcal{U}:= \pi^{-1}(\mathcal{N})$. By the third bullet in Def. \ref{simpledefi} we can assume, without loss of generality, that every fiber of $\pi$ contained in $\mathcal U\setminus\pi^{-1}(b)$ is connected. 

  Consider nets $(L'_i)_{i\in I}$ and $(L''_i)_{i\in I}$ of leaves contained in $\mathcal{U}$ such that $(i)$ $L'_i \stackrel{M/\f}{\longrightarrow} L' \subset \mathcal{U}$ and $L''_i \stackrel{M/\f}{\longrightarrow} L'' \subset \mathcal{U}$, and $(ii)$ for each $i\in I$, either $L'_i=L''_i$ or there exists a future-directed transversely causal curve segment $\gamma_i:[0,1] \rightarrow \mathcal{U}$ starting at $L'_i$ and ending at $L''_i$. 

  Now, we have, for each such $\gamma_i$,
  \begin{eqnarray}
      h_{\pi\circ \gamma_i(t)}((\pi\circ \gamma_i)'(t),(\pi \circ \gamma_i)'(t)) &=& h_{\pi(\gamma_i(t))}(d\pi_{\gamma_i(t)}(\gamma_i'(t)),d\pi _{\gamma_i(t)}(\gamma_i'(t))) \nonumber \\
      &=& (\pi^\ast h)_{\gamma_i(t)}(\gamma_i'(t),\gamma_i'(t)) \equiv (g_\intercal)_{\gamma_i(t)}(\gamma_i'(t),\gamma_i'(t))\leq 0,\nonumber  \\
      &\text{ and }& \gamma_i'(t) \notin \ker d\pi_{\gamma_i(t)} \Rightarrow (\pi\circ \gamma)'(t)\neq 0, \quad \forall t \in [0,1], \nonumber
  \end{eqnarray}
  whence we conclude that $\pi \circ \gamma_i$ is a (future-directed) $h$-causal curve in $\mathcal{N}$, and thus that 
  $$\pi(L'_i):=b'_i\leq _{\mathcal{N}}b''_i:= \pi(L_i'').$$
  But we have $b_i'\rightarrow b':= \pi(L')$ and $b_i''\rightarrow \pi(L'')$ in $B$. Since the $h$-causal relation within the normal convex open set $\mathcal{N}$ is closed (see, e.g., \cite[Lemma 14.2]{oneillbook}) we conclude by Proposition \ref{propussub} that $L'\leq_{M/\f} L''$.
\end{proof}

We can relate transverse stable causal with the ordinary stable causality of $(B,h)$ in this context as follows. 
\begin{corollary}\label{submersionstable}
  For the simple foliation $(M,\f,g_\intercal)$,
  $$(M,\f, g_\intercal) \text{ is transversely stably causal if and only if $ (B,h)$ is stably causal}.$$
\end{corollary}
\begin{proof}
  Just observe that given any transverse Lorentzian metric $g'_\intercal$ on $(M,\f)$, there is a unique Lorentz metric $h'$ on $B$ such that $g'_\intercal=\pi^\ast h'$. It is then easy to see that 
  $$g_\intercal < g'_\intercal \Longleftrightarrow h<h',$$
  and by Prop. (\ref{propussub})
  $$(M,\f, g'_\intercal) \text{ is transversely causal if and only if $ (B,h')$ is causal},$$
  whence the conclusion follows. 
\end{proof}
The proof of the previous corollary and the Corollary \ref{massamassa} also easily establish:
\begin{corollary}\label{neat}
    The transverse Seifert relation of a simple foliation $(M,\f,g_\intercal)$ is closed. Hence, any transversely stably causal simple foliation is transversely $K$-causal.
\end{corollary}
\qcd
The next result follows immediately from Corollary \ref{submersionstable} and \cite[Thm. 3.43]{beem}.
\begin{corollary}\label{submersionstablespecific}
    With the notation and assumptions as in Prop. \ref{propussub}, if $B$ is homeomorphic to $\mathbb{R}^2$, then $(M,\f, g_\intercal)$ is transversely stably causal.  
\end{corollary}
\qcd
\begin{exe}
In Corollary \ref{submersionstable} the assumption of connectedness of the fibers cannot be dropped in the `only if' part.
\emph{
   We know (conf. Example \ref{submersionsexe}) that for any onto submersion the stable causality of $(B,h)$ does still imply the transverse stable causality of $(M,\f,g_\intercal)$. But the converse fails in the general case. To see this, let $\pi: (t,x,y) \in \mathbb{R}^3 \mapsto (t,x) \in \mathbb{R}^2$. Consider on $\mathbb{R}^2$ the Minkowski (flat) metric 
   $$\eta = -dt^2 + dx^2$$
   with the standard time-orientation, so that if $\f$ is the foliation on $\mathbb{R}^3$ given by the fibers of $\pi$, which are all connected, and $g_\intercal:=\pi^\ast \eta $. Under these conditions, $(\mathbb{R}^3,\f, g_\intercal)$ is transversely stably causal by Corollary \ref{submersionstablespecific}. Now, let $G$ be the subgroup of isometries of $(\mathbb{R}^2,\eta)$ generated by the isometry 
   $$(t,x) \mapsto (t+1,x).$$
   $G$ obviously acts freely and properly on $\mathbb{R}^2$, so the quotient $B:=\mathbb{R}^2/G$ is a smooth manifold diffeomorphic to a $2d$ cylinder, and the quotient map $\varphi: \mathbb{R}^2 \rightarrow B$ is a covering map. Moreover, $\eta$ induces on $B$ a flat Lorentzian metric $h$ with closed timelike curves that ``are circles that go around the axis of the cylinder'', so that in particular $(B,h)$ is not even chronological\footnote{Indeed, this is a well-known example of a totally vicious spacetime, i.e., there is a closed timelike curve through each point.}, let alone (stably) causal. However, if we now consider the submersion $\hat{\pi}:=\varphi\circ \pi$, then the foliation $\hat{\f}$ is given by the connected components of its fibers, so we have still have $g_\intercal = \hat{\pi}^\ast h $, and $\hat{\f}=\f$. Nevertheless, this does not contradict Corollary \ref{submersionstable} because each fiber has infinitely many connected components, and there are transversely timelike curves connecting any two of these components projecting onto closed timelike curves on $(B,h)$. 
   }
\end{exe}

\medskip

\begin{proposition}\label{sub3-4}
If the simple foliation $(M,\mathcal{F}, g_\intercal)$ admits a transverse time function, then it admits a transverse temporal function. In particular, $(M,\mathcal{F}, g_\intercal)$ is transversely stably causal. 
\end{proposition}
\begin{proof}
    Let $f:M\rightarrow\mathbb{R}$ be a transverse time function. Since $f$ is constant on the leaves of $\mathcal{F}$, and given that $\pi$ satisfies the connectivity condition mentioned above, it is easy to check that $f$ is also constant along the fibers of $\pi$ and therefore there exists a continuous function $\hat{f}:B\rightarrow \mathbb{R}$ such that $\hat{f}\circ\pi=f$. Let $\gamma:[a,b]\rightarrow B$ be a smooth future-directed causal curve, and denote by $\hat{\gamma}$ some lift of $\gamma$ to $M$. Then, $g_{\intercal\hat{\gamma}(t)}(\hat{\gamma}^\prime(t),\hat{\gamma}^\prime(t))=h_{\gamma(t)}(\gamma^\prime(t),\gamma^\prime(t))\leq 0$, since $\pi$ is, up to completing the vertical part of $g_\intercal$, a Lorentzian submersion. This means that $\hat{\gamma}$ is transversely causal and therefore $f$ increases along it. We thus conclude that $\hat{f}$ is a time-function. By \cite{Bernal_2005}, there exists a smooth temporal function $\tilde{f}: B\rightarrow \mathbb{R}$. Then, $\tilde{f}\circ\pi$ is a transverse temporal function on $M$.
\end{proof}

\begin{proposition}\label{sub1-3}
If a simple foliation $(M,\mathcal{F}, g_\intercal)$ is transversely stably causal, then it admits a transverse time function.
\end{proposition}
\begin{proof}
By Corollary \ref{submersionstable}, $(B,h)$ is stably causal. Then $(B,h)$ admits a continuous time function $f$, from which we obtain a transverse time function $f\circ \pi$ on $(M,\mathcal{F},g_\intercal)$.
\end{proof}



\begin{lemma}\label{a-necessary-lemma}
    Let $\hat{x}, \hat{y}\in M/\mathcal{F}$. Then, $\hat\pi(\hat x)=\hat\pi(\hat y)$, if and only if $(\hat{x},\hat{y})\in\overline{\Delta_{M/\f}}$.
\end{lemma}
\begin{proof}
$(\implies)$ It follows easily from the special property which is being required for the submersion $\pi$ that every neighborhood of $\hat{x}$ and $\hat{y}$ must intersect.\\
$(\impliedby)$ If $\hat\pi(\hat x)\neq\hat\pi(\hat y)$, $B$ being Hausdorff we can construct $\pi-$saturated disjoint neighborhoods $\mathcal{U}, \mathcal{V}\subset M$ respectively containing the leaves of $\hat x$ and $\hat{y}$. The result follows by contradiction once one notes that any set which is $\pi-$saturated  is also $\f-$saturated.
\end{proof}

\begin{lemma}\label{lookatthat}
Let $R_\intercal$ be any closed relation on $M/\f$ containing the relation $J^+_{M/\f}$. Then, $\overline{\Delta_{M/\f}}\subset R_\intercal$. In particular, $\overline{\Delta_{M/\f}}\subset K^+_\intercal$.
\end{lemma}\begin{proof}
\begin{align*}
    \Delta_{M/\f}&\subset J^+_{M/\f}\subset R_\intercal\\
    &\implies\\
    \overline{\Delta_{M/\f}}&\subset \overline{J^+_{M/\f}}\subset \overline{R_\intercal}=R_\intercal
\end{align*}
\end{proof}

\begin{proposition}\label{sub6-1}
If a simple foliation $(M, \mathcal{F}, g_\intercal)$ is transversely $K$-causal, then it is transversely stably causal.
\end{proposition}
\begin{proof}

It suffices to see that $K^+(h)$ is antisymmetric, because then $(B,h)$ is stably causal and the result follows from Corollary \ref{submersionstable}. To that end, we will consider the pushforward relation $\hat{\pi}_*K^+_\intercal\subset B\times B$, and we will check that $\hat{\pi}_*K^+_\intercal$ is closed, transitive and contains $J^+(h)$.
\begin{itemize}
    \item Let $(\hat{\pi}(\hat{a}_k),\hat{\pi}(\hat{b}_k))_{k\in\mathbb{N}}\subset \hat{\pi}_*K^+_\intercal$ converge to $(\tilde{x},\tilde y)\in B\times B$. For each $k\in\mathbb{N}$ we have $(\hat{a}_k,\hat{b}_k)\in K^+_\intercal$. Choose any $(\hat{a},\hat{b})\in M/\f\times M/\f$ such that $\hat{\pi}(\hat{a})=\tilde x$ and $\hat{\pi}(\hat b)=\tilde y$. Fix an open neighborhood $\tilde x\subset\mathcal{U}$ such that all the fibers of $\pi$ inside $\pi^{-1}(\mathcal{U})\setminus\pi^{-1}(\tilde x)$ are connected. Shrinking $\mathcal{U}$, if necessary, we can assume there is a section $s:\mathcal{U}\rightarrow M/\f$ of $\hat{\pi}$ such that $s(\tilde x)=\hat{a}$. For each $k\in\mathbb{N}$ sufficiently large so that $\hat{\pi}(\hat{a}_k)$ enters $\mathcal{U}$ and no longer leaves it, define $\hat\alpha_k:=(s\circ\hat{\pi})(\hat{a}_k)$. It is clear that $\hat{\alpha}_k\to\hat a$ and $(\hat{\alpha}_k,\hat a_k)\in\overline{\Delta_{M/\f}}, \forall k$. Analogously, we obtain a sequence $\hat{\beta}_k$ which converges to $\hat b$ and $(\hat{\beta}_k,\hat b_k)\in\overline{\Delta_{M/\f}}, \forall k$. Then, for each $k$, it follows that both $(\hat{\alpha}_k,\hat a_k)$ and $(\hat b_k, \hat{\beta}_k)$ are in $K^+_\intercal$. Applying the transitivity of $K^+_\intercal$ twice, it yields $(\hat{\alpha}_k,\hat\beta_k)\in K^+_\intercal, \forall k$. Since $K^+_\intercal$ is closed, it follows that $(\hat a,\hat b)\in K^+_\intercal$, whence $(\tilde x, \tilde y)=(\hat{\pi}(\hat{a}), \hat{\pi}(\hat{b}))\in \hat{\pi}_*K^+_\intercal$, establishing that $\hat{\pi}_*K^+_\intercal$ is closed.
    \item Suppose $(\hat{\pi}(a),\hat{\pi}(b))$ and $ (\hat{\pi}(b),\hat{\pi}(c))$ are in $\hat{\pi}_*K^+_\intercal$. Then, $(a, b), (b^{\prime},c)\in K^+_\intercal$, where $\hat{\pi}(b)=\hat{\pi}(b^\prime)$. By Lemma \ref{lookatthat}, we have that $(b,b^\prime)\in K^+_\intercal$ and then $(a,c)\in R_\intercal$, whence $(\hat{\pi}(a),\hat{\pi}(c))\in\hat{\pi}_*K^+_\intercal$.
    \item Let $(a,b)\in J^+(h)$. Then, there exists a future-directed $h$-causal curve $\alpha:[0,1]\rightarrow B$ such that $\alpha(0)=a$ and $\alpha(1)=b$. We can easily obtain a $\hat{\pi}$-lift of $\alpha$ to $M/\f$ (lift it first to $M$, then project with $\pi_\f$), which yields that $(\hat{a},\hat{b})\in J^+_{M/\f}$, for some $\hat{a}\in\hat{\pi}^{-1}(a)$ and some $\hat{b}\in\hat{\pi}^{-1}(b)$. Then, $(\hat{a},\hat{b})\in K^+_\intercal$, whence $(a,b)=(\hat{\pi}(\hat{a}), \hat{\pi}(\hat{b}))\in\hat{\pi}_*K^+_\intercal$. We conclude that $\hat{\pi}_*K^+_\intercal\supset J^+(h)$.
\end{itemize}
Then $\hat{\pi}_*K^+_\intercal\supset K^+(h)$ and therefore, given $x,y \in B$ such that $(x,y),(y,x)\in K^+(h)$, we have that $(x,y),(y,x)\in \hat{\pi}_*K^+_\intercal$. That means that there are $\hat{x}_0, \hat{y}_0, \hat{x}_1, \hat{y}_1\in M/\f$ such that $\hat{\pi}(\hat{x}_0)=x=\hat{\pi}(\hat{x}_1), \hat{\pi}(\hat{y}_0)=y=\hat{\pi}(\hat{y}_1)$ and $(\hat{x}_0,\hat{y}_0),(\hat{y}_1,\hat{x}_1)\in K^+_\intercal$. By Lemmas (\ref{a-necessary-lemma}) and (\ref{lookatthat}) we have that $(\hat{x}_1,\hat{x}_0)$ and $(\hat{y}_0, \hat{y}_1)$ are in $K^+_\intercal$. Employing the transitivity of $K^+_\intercal$ twice we obtain that $(\hat{y}_0,\hat{x}_0)\in K^+_\intercal$, and since $(M,\f, g_\intercal)$ was assumed to be transversely $K$-causal, we have that $(\hat{x}_0,\hat{y}_0)\in\overline{\Delta_{M/\f}}$. By Lemma (\ref{a-necessary-lemma}), it follows that $x=\hat{\pi}(\hat{x}_0)=\hat{\pi}(\hat{y}_0)=y$, which establishes the antisymmetry of $K^+(h)$.
\end{proof}

\begin{proposition}\label{sub2-1}
    Let $(M, \f, g_\intercal)$ be a simple foliation. If $\pi$ is proper and there is a $C^0$ neighborhood $\mathcal{U}$ of $g_\intercal$ in $\mathcal{TL}(M,\mathcal{F})$ such that for all $g^\prime_\intercal\in\mathcal{U}$ we have that $(M, \mathcal{F}, g^\prime_\intercal)$ is transversely causal, then $(M, \f, g_\intercal)$ is transversely stably causal.
\end{proposition}
\begin{proof}
Consider the map $\Psi: \mathrm{Lor}(B)\rightarrow \mathcal{TL}(M,\mathcal{F})$ given by $\Psi(l)=\pi^*l$. It is clearly a bijection with inverse given by a natural projection. Since $\pi$ is proper, we have that $\Psi$ is continuous, since $\Psi$ maps $l$ to its pullback with $\pi$, and such can be written as a composition of $l$ with $\pi$ and $d\pi,$  and such compositions are continuous with respect to the strong Whitney topologies iff $\pi$ is proper. Then, $\mathcal{V:=}\Psi^{-1}(\mathcal{U})$ is a open neighborhood of $h$. Pick any $h^\prime\in \mathcal{V}$. If there was any $h^\prime$-causal curve $\gamma:[a,b]\rightarrow B$ such that $\gamma(a)=\gamma(b)$, then, any lift of $\gamma$ to $M$ would yield a $\Psi(h^\prime)$-transversely causal curve starting and ending at the same leaf. This cannot happen, since 
$\Psi(h^\prime)$ is in $\mathcal{U}$. Then, $(B, h^\prime)$ is causal for every $h^\prime\in \mathcal{V}$. The classical theory then implies the existence of a Lorentzian metric $h^*$ on $B$ with $h<h^*$ and $(B, h^*)$ causal. If we set $g_\intercal^*:=\Psi(h^*)$, then it is clear that $g_\intercal<g^*_\intercal$ and $(M, \mathcal{F}, g^*_\intercal)$ is transversely causal.
\end{proof}

\begin{figure}[h]
\centering{
\scalebox{0.8}{
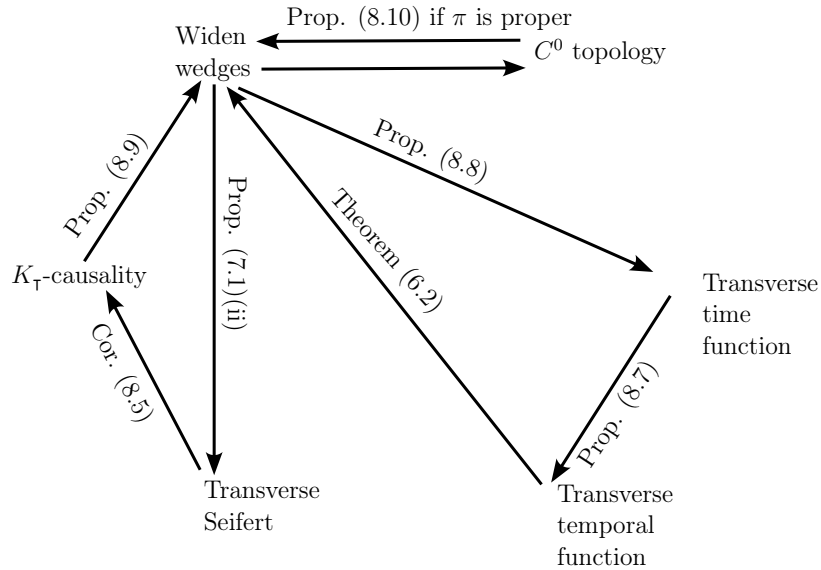}
\caption{Diagram of the logical implications established for simple foliations.}
\label{subequivalences}}
\end{figure}

\section{Remarks on a transverse causal hierarchy}\label{sec:causalhierq}

We end this paper with brief remarks on how the novel concepts discussed so far fit in the ``transverse causal hierarchy'' of Lorentzian foliations as initially and partially developed in \cite{us}. Of course, a major goal would be to prove some version of\begin{align*}
\text{transversely} &\text{ causally continuous}\\ &\Downarrow \\\text{transversely}&\text{ stably causal}\\&\Downarrow \\\text{transversely}&\text{ strongly causal},
\end{align*}exactly as we have in the spacetime case. However, at the moment the notion of transverse causal continuity has not yet been developed, and the argument for the spacetime version of the second implication relies rather heavily on the existence of normal convex neighborhoods, a transverse analogue of which is not yet at our disposal.

Nevertheless, upon the assumption of either causal leaf-regularity as in Definition \ref{regulardefi} or for simple foliations we can give partial results as we presently discuss. Recall (see \cite{us} for further details) that a Lorentzian foliation $(M,\f, g_\intercal)$ satisfies the \textit{transverse strongly causality condition} at a leaf $L\in \f$ if given any saturated open set $L\subset \mathcal{U} \subset M$ we can find another saturated open set $L\subset \mathcal{V}\subset \mathcal{U}$ in $M$ such that any transversely causal curve segment in $M$ with endpoints in $\mathcal{V}$ is entirely contained in $\mathcal{U}$. We then say $(M,\f,g_\intercal)$ itself is \textit{transversely strongly causal} if the transverse strong causality holds at each leaf. 
We now have:
\begin{teo}\label{leafregular}
Suppose the Lorentzian foliation $(M,\f, g_\intercal)$ is causally leaf-regular and its leaf space $M/\f$ is Hausdorff. If $(M,\f, g_\intercal)$ is transversely stably causal, then it is also transversely strongly causal. 
\end{teo}
\begin{proof}
Suppose the transverse strong causality condition fails at some leaf $L\in \mathcal F$. Then there is a saturated open set $L\subset \mathcal U$ in $M$ and a net $(\gamma_i:[0,1] \rightarrow M)_{i\in I}$ of future-directed, $g_\intercal$-transversely causal curves such that $L_{\gamma_i(0)},L_{\gamma_i(1)}\rightarrow L$ in $M/\f$, and for each $i\in I$ there is a $0<t_i<1$ so that $\gamma_i(t_i) \notin \mathcal{U}$. Let $\mathcal{B}$ denote an open saturated neighborhood of $L$ in $M$ with ``closed transverse causal relation'' as warranted by causal leaf-regularity.

Observe that the leaf space is locally compact: given a leaf $L'$, one can pick a compact set $K$ such that $(int _M\, K)\cap L'\neq \emptyset$, and since the projection $\pi_\f:M\rightarrow M/\f$ is an open, continuous map, we have that 
$$L' \in \pi_\f(int_M\, K)\subset \pi_\f(K),$$
thus proving the claim. Since we are also assuming that $M/\f$ is Hausdorff, we can pick a relatively compact open set $\mathcal{O}\ni \pi_\f(L)$ in $M/\f$ such that $\overline{\mathcal O}^{M/\f}\subset \pi_\f(\mathcal{B}\cap\mathcal{U})$. We eventually have $\gamma_i(0),\gamma_i(1) \in \pi_\f^{-1}(\mathcal{O})\subset \mathcal{U}$, and we can assume by passing to a subnet if necessary that this is always the case.

Now, for each $i\in I$, the fact that $\gamma_i(t_i) \notin \mathcal{U}\equiv \pi_\f^{-1}(\pi_\f(\mathcal{U))}$ means that $\pi_\f(\gamma_i(t_i)) \notin \pi_\f(\mathcal{U})$, and hence the continuous curve $\pi_\f\circ \gamma_i:[0,1] \rightarrow M/\f$ leaves $\mathcal{O}$, that is, for each $i\in I$ we have numbers $0<s_i<t_i<s'_i<1$ such that $(1)$ $\pi_\f(\gamma_i(s_i)),\pi_\f(\gamma_i(s'_i)) \in \partial_{M/\f} \mathcal{O}$ and $(2)$ $\pi_\f\circ \gamma_i[0,s_i),\pi_\f \circ\gamma_i(s'_i,1] \subset \mathcal{O}$. By the compactness of $\partial _{M/\f}\mathcal{O}$ we again can assume, again up to passing to a subnet, that $\pi_\f(\gamma_i(s_i))\rightarrow \hat{x},\pi_\f(\gamma_i(s'_i)\rightarrow \hat{y}$ in $M/\f$ with $\hat{x},\hat{y}\in \partial_{M/\f}\mathcal{O}$. 

However, the fact that $\gamma_i[0,s_i], \gamma_i[s'_i,1] \subset \pi_\f^{-1}(\mathcal{O})\subset \mathcal{B}$ means - by causal leaf-regularity - that, passing to the limit, 
\begin{equation}\label{eqforladder1}
L<_{M/\f}L':= \pi_\f^{-1}(\hat{x}) \text{ and } \pi_\f^{-1}(\hat{y}) =: L''<_{M/\f}L. \end{equation}

Observe next that given any transverse Lorentz metric $g_\intercal <\tilde g_\intercal$ we must have that the curve $\gamma_i|_{[s_i,s_i']}$ is $\tilde g_\intercal$-transversely timelike; thus, on the one hand, we can conclude from the arbitrariness of $\tilde g_\intercal$ that 
$$(L_{\gamma_i(s_i)},L_{\gamma_i(s'_i)}) \in J^+_{S\intercal}, \forall i \in I,$$
and since the transverse Seifert relation is closed by Thm. \ref{theone}, we also have 
$$(L',L'') \in J^+_{S\intercal};$$
on the other hand, by \eqref{eqforladder1} we must have 
$$L\ll_{\tilde g_\intercal}L'\leq_{\tilde g_\intercal} L\ll_{\tilde g_\intercal} L.$$ 
But since $(M,\f, g_\intercal)$ is transversely stably causal, we can pick $\tilde{g}_\intercal$ in the previous relations so that $(M,\f, \tilde g_\intercal)$ is transversely causal, in contradiction.
\end{proof}

We also recall \cite[Def. 5.3]{us} that $(M,\f,g_\intercal)$ is \textit{transversely globally hyperbolic} if it is transversely strongly causal and the \textit{causal prisms}
$$J_\intercal(L, L'):= J^+_\intercal(L)\cap J^-_\intercal(L')$$
are transversely compact in $M$ for every $L,L'\in \f$. We thus have:
\begin{teo}
   If $(M, \f,g_\intercal)$ is simple, then 
   \begin{align*}
\text{transversely} &\text{ globally hyperbolic}\\ &\Downarrow \\\text{transversely}&\text{ stably causal}\\&\Downarrow \\\text{transversely}&\text{ strongly causal}.
\end{align*}
 
\end{teo}

\begin{proof}
We adopt here the notation and conventions in section \ref{sec:submersions}. By Corollary \ref{submersionstable}, to show the first implication it suffices to show that $(B,h)$ is stably causal. We do this by showing it is globally hyperbolic.

If an open set $\mathcal{U}\subset B$ is given, and $x\in \mathcal{U}$, then $\pi^{-1}(\mathcal{U})$ is an $\f$-saturated open set in $M$ containing $L_x$, and by the transverse strong causality condition at $L_x$ we have an $\f$-saturated open set $L_x\subset \mathcal{V}\subset \pi^{-1}(\mathcal{U})$ such that any transverse causal curve with endpoints in $\mathcal{V}$ lies entirely in $\pi^{-1}(\mathcal{U})$. Suppose $\gamma:[a,b]\rightarrow B$ is $h-$causal with $\gamma(a),\gamma(b)\in\pi(\mathcal{V})$. Now, fix some $p\in\pi^{-1}(\gamma(a))\cap\mathcal{V}$ and some $q\in\pi^{-1}(\gamma(b))\cap\mathcal{V}$ and consider a lift $\hat\gamma:[a,b]\rightarrow M$ of $\gamma$ starting at $p$, which is transversely timelike. Since the submersion $\pi$ is as in section \ref{sec:submersions}, by a repeated application of the Causal Waterfall lemma, we can deform $\hat\gamma$ in the vertical directions in order to ensure that $\hat\gamma(b)=q$. Then, $\hat{\gamma}$ lies entirely inside of $\pi^{-1}(\mathcal{U})$ and we conclude that $\gamma$ is contained in $\mathcal{U}$, thus establishing that $(B,h)$ satisfies strongly causality at $x$, and since the latter is arbitrary, everywhere.

Now, we show that the causal diamonds of $(B,h)$ are compact. Indeed, let $p, q\in B$ with $p<_hq$ and denote $J(p,q)=J^+(p)\cap J^-(q)$. Suppose $(x_k)_{k\in\mathbb{N}}\subset J(p,q)$ is any sequence. Let $\hat{p}\in\pi^{-1}(p)$ and $\hat q\in \pi^{-1}(q)$, and for each $k\in\mathbb{N}$ choose some $\hat{x}_k\in\pi^{-1}(x_k)$. A simple lift argument shows that each $\hat{x}_k$ is contained in the causal prism $J_\intercal(\hat{p}, \hat{q})$, which is transversely compact by hypothesis. Then, up to passing to a subsequence, we can assume $(\hat{x}_k)_{k\in\mathbb{N}}$ converges in $M/\f$ to a leaf $L\in J_\intercal(\hat{p}, \hat{q})$, and by the continuity of $\pi$ we conclude that some subsequence of $(x_k)_{k\in\mathbb{N}}$ converges to some $x\in J(p,q) \equiv \pi(J_\intercal(\hat{p}, \hat{q}))$.

The proof of the second implication is analogous but much simpler, thus we omit it. 
\end{proof}

\section*{Acknowledgments}
This study was supported by the Fundacao de Amparo a Pesquisa e Inovacao do Estado de Santa Catarina (FAPESC) process 985/2025.

\appendix

\section{The $C^r$ topology for $\mathcal{TL}^r(M,\mathcal{F})$ } \label{appendixA}
In this appendix we expand on the technical details of the $C^r$ topology over $\mathcal{TL}^r(M,\mathcal{F}).$

First, note  that $\mathcal{TL}^r(M,\mathcal{F})$ is a subset of the broader space of symmetric $(0,2)$-type tensor fields of class $C^r$ on $M$. We denote as $\mathrm{Sym}^{(0,2)}(M)$ the vector bundle of such tensors. Our aim here is to identify $\mathcal{TL}^r(M,\mathcal{F})$ with a subspace of sections of some ``well behaved'' subbundle of $\mathrm{Sym}^{(0,2)}(M).$ 

Given any smooth fiber bundle $\pi: E\to M$, we denote it simply by its total space $E$ if there is no risk of confusion, and by $\Gamma^r(E)$ we denote the space of the $C^r$ sections of $E$. Consider the subset  $\mathcal S\subset \mathrm{Sym}^{(0,2)}(M)$ given as follows: $g \in \mathcal S$ iff for all $x \in M,$  
$$T_x \mathcal F \subseteq \ker \, g_x:= \{ v \in T_xM \, : \, g_x(v,w) = 0, \forall w \in T_xM\}.$$

Given a foliated chart $(\mathcal{U}, \varphi=(x^1,\ldots,x^n))$ of $(M, \mathcal F)$ we adopt the convention that the plaques of $\mathcal{F}$ are the fibers of $(x^1,\ldots, x^q)$, and its transversals are the fibers of $(x^{q+1}, \ldots, x^n)$, so that, for each $a=1, \ldots, q$ we have that $\partial/\partial x^a$ is transverse to the leaves, whereas for each $i=q+1,\ldots, n$ the coordinate vector field $\partial/ \partial x^i$ is tangent to the leaves. For $g \in \mathcal S,$ we always have
\[ \begin{aligned}
	g_{\mu i}=g_{i\mu} &= g\left(\frac{\partial}{\partial x^i},\frac{\partial}{\partial x^\mu}\right) = 0, \quad q+1\leq i\leq n, 1\leq \mu\leq n.
\end{aligned}
\]
Denote by $S^{(0,2)}\varphi$ the local trivialization of the bundle $\mathrm{Sym}^{(0,2)}(M)$ induced by the chosen coordinate chart. We see that if  $g \in \mathcal S$, then, for all $x \in M,$ $S^{(0,2)}\varphi(g_x)$ has the form $(x, G(x)),$ where $G(x)$ is the symmetric $n\times n$ matrix
\begin{equation}\label{eq:matrix1}
	\left[
	\begin{tabular}{l|l}
		$[g_{ab}(x)]_{q\times q}$ & $0_{q\times (n-q)}$        \\  \hline
		$0_{(n-q)\times q}$ & $0_{(n-q)\times (n-q)}$
	\end{tabular}
	\right].
\end{equation}
It follows that $\mathcal S$ is a closed vector subbundle of  $\mathrm{Sym}^{(0,2)}(M).$ We consider further the subbbundle $\mathcal N_1({\mathcal S})\subset \mathcal S$ of bilinear forms $g \in \mathcal S$ for which, in local coordinates as above, the submatrix  $[g_{ab}]$ in (\ref{eq:matrix1}) is non-singular and has index equal to 1 (that is, its associated bilinear form on $\mathbb R^{q}$ has this property). This set of matrices is open in the set of symmetric $q \times q$ matrices, and therefore $\mathcal N_1({\mathcal S})$ is actually an open subbundle of $\mathcal S.$ 

Consider now the space of sections $\Gamma^r (\mathcal S)$ with the Whitney strong $C^r$ topology. The space  $\Gamma^r(\mathcal N_1(\mathcal S)),$ which will be denoted by $\mathrm{Lor}^r(\mathcal S)$ ($\mathrm{Lor}(\mathcal S)$ for $r = \infty$) can be embedded into $\Gamma^r(\mathcal S)$ as a $C^r$-open subset (this follows by adapting Proposition 9.2.13 in  \cite{margalef_differential_1992} together with the standard definition of the Whitney topologies in terms of jet bundles). Now, $\Gamma^r(\mathcal S)$ is closed in the Whitney \textit{weak} $C^r$ topology on $C^r(M, \mathcal S)$, and thus it is a \textit{Baire space}, i.e., countable intersections of dense open sets is dense (conf. \cite{hirsch_differential_1976}, Theorem 4.4 for finite $r$; for the $C^\infty$ topology one can easily adapt the proof in \cite{cinfbaire} for the Baire property, adding the condition of a weakly $C^\infty$ subset). Hence, $\mathrm{Lor}^r(\mathcal S)$ is also a Baire space, since it is open in $\Gamma^r(\mathcal S).$ 

The space $\mathrm{Lor}^r(\mathcal S)$ is not yet exactly what we want, since tensor fields in $\mathrm{Lor}^r(\mathcal S)$ need not be holonomy-invariant. Nevertheless, $\mathcal{TL}^r(M,\mathcal{F})$ is a subset of $\mathrm{Lor}^r(\mathcal S).$ Now, for each fixed $V \in  \mathfrak{X}(\mathcal F),$ the function
\[
\begin{aligned}
	\mathcal L_V : \mathrm{Lor}^r(\mathcal S) &\to \Gamma^{r-1}(\mathrm{Sym}^{(0,2)}(M))\\
	g &\mapsto \mathcal L_V g
\end{aligned}
\]
is continuous with respect to the weak $C^r$ topology on $\mathrm{Lor}^r(\mathcal S)$ and the weak $C^0$ topology on the codomain (see \cite{lerner72}, 3.4 in p. 20 for the strong topology case; the proof for the weak topology is completely analogous), therefore $\left(\mathcal L_V\right)^{-1}(0)$ is $C^r$ weakly closed in   $\mathrm{Lor}^r(\mathcal S).$ Clearly
\[
\mathcal{TL}^r(M,\mathcal{F}) = \bigcap_{V \in \mathfrak{X}(\mathcal F) } \left(\mathcal L_V\right)^{-1}(0) \subseteq \mathrm{Lor}^r(\mathcal S),
\] 
and $\mathcal{TL}^r(M,\mathcal{F})$ is weakly closed in $\mathrm{Lor}^r(\mathcal S)$, whence it is a Baire space in the strong Whitney $C^r$ topology as well. We have proved: 

\begin{proposition}\label{smallA}
	For each $1\leq r\leq\infty$ the set $\mathcal{TL}^r(M,\mathcal{F})$ of transverse Lorentzian metrics on $(M,\f)$ of class $C^r$ is a Baire space in the strong Whitney $C^r$ topology.
\end{proposition}
\qcd

\section{$\mathcal{TL}^r(M,\mathcal{F})$ and   $\mathrm{Lor}^r_\intercal(\nu\mathcal F)$ are homeomorphic} \label{appendixB} 

For $\nu\f=TM/ T\mathcal F$, consider the vector bundle morphism
\[
\begin{aligned}
	\Phi : 	\mathcal S &\to \mathrm{Sym}^{(0,2)}(\nu\mathcal F)\\
	b &\mapsto \bar{b}, 
\end{aligned}
\]
where $\bar{b}(\bar{v},\bar{v}) = b(v,v).$ It is easy to see the independence of the choice of $v \in T_xM,$ since if $v \sim v',$ then $v - v'\in T_x \mathcal F_x.$ By writing a local version of  $\Phi$ with local trivializations induced by a foliated atlas, the smoothness of $\Phi$ is clear. It is also a surjection fiberwise, and therefore a vector bundle isomorphism, since $\mathcal S$ and $\mathrm{Sym}^{(0,2)}(\nu\f)$ have the same rank. Given that the $C^r$ sections over a vector bundle form a $C^r$ module for the topological ring $C^r(M)$ with $C^r$ topology  (\cite{mather1969infinitesimal} section 2 corollary 2), $\Phi$ induces a topological module isomorphism

\[
\begin{aligned}
	\Phi_\# : 	\Gamma^r(\mathcal S) &\to \Gamma^r(\mathrm{Sym}^{(0,2)}(\nu\mathcal F))\\
	g & \mapsto \Phi \circ g.  
\end{aligned}
\]
Denoting by $\mathcal N_1(\nu\mathcal F) \subseteq \mathrm{Sym}^{(0,2)}(\nu\mathcal F)$ the open subbundle of fiberwise Lorentzian scalar products and by $\mathrm{Lor}^r(\nu \mathcal F)$ the space of $C^r$  sections of $\mathcal N_1(\nu \mathcal F)$  (again, denoted $\mathrm{Lor}(\nu \mathcal F)$ for $r = \infty$),   $\Phi_\#$ restricts to a homeomorphism $	\Phi_\# : 	\mathrm{Lor}^r(\mathcal S) \to \mathrm{Lor}^r(\nu\mathcal F).$ 

Recall that,  for  $ B \in \mathrm{Sym}^{(0,2)}(\nu\mathcal F)$ the Lie derivative in the direction of $V \in \mathfrak{X}(\mathcal F)$ is
\[
\mathcal L_V B(\overline X, \overline Y) =  V(B(\overline X, \overline Y)) -  B(\mathcal L_V \overline X, \overline Y) - B(\overline X, \mathcal L_V \overline  Y), \quad \overline X, \overline Y \in \Gamma^r(\nu\mathcal F),
\]
where $\mathcal L_V \overline X = \overline{[V,X]},$ where we have defined the holonomy-invariant $(0,2)$-symmetric tensor fields of $\nu\mathcal F$ as those $B \in \Gamma^r(\mathrm{Sym}^{(0,2)}(\nu\mathcal F))$ such that $\mathcal L_V B = 0$ for all $V \in \mathfrak{X}(\mathcal F),$ and we denoted $\mathrm{Lor}^r_\intercal(\nu\mathcal F)$ the $C^r$ holonomy-invariant Lorentzian fiber metrics of $\nu\mathcal F$. Further restricting $\Phi_{\#}$ to $\mathcal{TL}^r(M,\mathcal{F})$ with its induced $C^r$ topology, this restriction is then a homeomorphism onto $\mathrm{Lor}^r_\intercal(\nu\mathcal F)$.

\newpage

\printbibliography

@book{hirsch_differential_1976,
	address = {New York, NY},
	series = {Graduate {Texts} in {Mathematics}},
	title = {Differential {Topology}},
	volume = {33},
	publisher = {Springer New York},
	author = {Hirsch, Morris W.},
	year = {1976},
	
}

@book{wald_general_1984,
	address = {Chicago},
	title = {General relativity},
	publisher = {University of Chicago Press},
	author = {Wald, Robert M.},
	year = {1984},
}

@book{margalef_differential_1992,
	address = {Amsterdam},
	title = {Differential {Topology}},
	publisher = {North-Holland},
	author = {Margalef-Roig, J. and Outerelo Domínguez, E.},
	year = {1992},
}

@article{lerner,
	title = {The space of {Lorentz} metrics},
	volume = {32},
	number = {1},
	journal = {Communications in Mathematical Physics},
	author = {Lerner, David E.},
	year = {1973},
	pages = {19--38},
	doi = {10.1007/BF01646426}
}

@thesis{lerner72,
title = {The space of Lorentz metrics on a non-compact manifold},
institution = {University of Pittsburgh},
author = {Lerner, David E.},
year = {1972},
type  = {PhD dissertation},
}

@misc{caramello2024transversegeometrylorentzianfoliations,
	title={Transverse Geometry of Lorentzian foliations with applications to Lorentzian orbifolds}, 
	author={Francisco C. Caramello Jr and Henrique A. Puel Martins and Ivan P. Costa e Silva},
	year={2024},
	eprint={2402.05907},
	archivePrefix={arXiv},
	primaryClass={math.DG},
	url={https://arxiv.org/abs/2402.05907}, 
}

@article{mather1969infinitesimal,
	title={Stability of $C^\infty$ mappings II. Infinitesimal stability implies stability},
	author={Mather, JN},
	journal={Ann. of Math},
	volume={89},
	pages={254--291},
	year={1969}
}

@misc{caramello3,
      title={Introduction to orbifolds}, 
      author={Caramello, Jr., F. C.},
      year={2022},
      eprint={1909.08699},
      archivePrefix={arXiv},
      primaryClass={math.DG},
      url={https://arxiv.org/abs/1909.08699}, 
}

@book{oneillbook,
    author = {O'Neill, B.},
    title = {Semi-Riemannian Geometry with Applications to Relativity},
    publisher = {Academic Press},
    year = {1983},
adress = {New York}
}

@book{beem,
    author = {Beem, J. K. and Ehrlich, P. E. and Easley, K.},
    title = {Global Lorentzian geometry},
    publisher = {Marcel Dekker Inc.},
    year = 1996,
edition = {2nd},
address = {New York}
}

@article{minguzzi-sanchez,
    author = {Minguzzi, E. and Sánchez, M.},
    title = {The causal hierarchy of spacetimes},
    journal = {ESI Lectures in Mathematics and Physics},
    year = 2008,
publisher = {European Mathematical Society Publ. House},
address = {Zürich}
}

@book{candel,
    author = {Candel, A. and Conlon L.},
    title = {Foliations I},
    publisher = {American Mathematical Society},
    year = {2000}
}

@book{hawking-ellis,
    author = {Hawking, S. and Ellis, G. F. R.},
    title = {The large scale structure of space-time},
    publisher = {Cambridge Monographs on Math. Physics},
    year = {1973}
}

@book{mrcun,
    author = {Moerdijk, I. and Mr\v{c}un, J.},
    title = {Introduction to Foliations and Lie Groupoids},
    publisher = {Cambridge Univerisity Press},
    year = {2003}
}

@book{molino,
    author = {Molino, P.},
    title = {Riemannian Foliations},
    publisher = {Birkhäuser},
    year = {1988}
}

@misc{us,
      title={Transverse Geometry of Lorentzian foliations with applications to Lorentzian orbifolds}, 
      author={Caramello, Jr., F.C. and Puel Martins, H. A. and Costa e Silva, I. P.},
      year={2024},
      eprint={2402.05907},
      archivePrefix={arXiv},
      primaryClass={math.DG},
      url={https://arxiv.org/abs/2402.05907}, 
}

@article{Minguzzi_2009,
   title={K-Causality Coincides with Stable Causality},
   volume={290},
   ISSN={1432-0916},
   url={http://dx.doi.org/10.1007/s00220-009-0794-4},
   DOI={10.1007/s00220-009-0794-4},
   number={1},
   journal={Communications in Mathematical Physics},
   publisher={Springer Science and Business Media LLC},
   author={Minguzzi, E.},
   year={2009},
   month=apr, pages={239–248} }

@article{Minguzzi_2007a,
   title={The causal ladder and the strength of K-causality: I},
   volume={25},
   ISSN={1361-6382},
   url={http://dx.doi.org/10.1088/0264-9381/25/1/015009},
   DOI={10.1088/0264-9381/25/1/015009},
   number={1},
   journal={Classical and Quantum Gravity},
   publisher={IOP Publishing},
   author={Minguzzi, E.},
   year={2007},
   month=dec, pages={015009} }

@article{Minguzzi_2007,
   title={The causal ladder and the strength of K-causality: II},
   volume={25},
   ISSN={1361-6382},
   url={http://dx.doi.org/10.1088/0264-9381/25/1/015010},
   DOI={10.1088/0264-9381/25/1/015010},
   number={1},
   journal={Classical and Quantum Gravity},
   publisher={IOP Publishing},
   author={Minguzzi, E.},
   year={2007},
   month=dec, pages={015010} }

@article{Bernal_2005,
   title={Smoothness of Time Functions and the Metric Splitting of Globally Hyperbolic Spacetimes},
   volume={257},
   ISSN={1432-0916},
   url={http://dx.doi.org/10.1007/s00220-005-1346-1},
   DOI={10.1007/s00220-005-1346-1},
   number={1},
   journal={Communications in Mathematical Physics},
   publisher={Springer Science and Business Media LLC},
   author={Bernal, Antonio N. and Sánchez, Miguel},
   year={2005},
   month=apr, pages={43–50} }

@article{molino2,
    author = {Molino, P.},
    title = {Feuilletages riemanniens sur les variétés compactes; champs de Killing transverses},
    journal = {Comptes Rendus de l'Académie des Sciences},
    year = {1974},
    volume = {289},
    pages = {421--423}
}

@article{russas,
    author = {Dolgonosova, A. Yu. and Zhukova, N. I.},
    title = {Pseudo-Riemannian foliations and their graphs},
    journal = {Lobachevskii J. of Mathematics},
    volume = 39,
    number = 1,
    year = 2018,
    pages = {54--64}
}

@misc{chrusciellowcausal1,
      title={Elements of causality theory}, 
      author={Chruściel, P. T.},
      year={2011},
      eprint={1110.6706},
      archivePrefix={arXiv},
      primaryClass={gr-qc},
      url={https://arxiv.org/abs/1110.6706}, 
}

@article{length2,
    author = {Burtscher, A. and Garcia-Heveling, L.},
    title = {Time functions on Lorentzian length spaces},
    journal = {Annales Henri Poincaré},
    year = 2024,
}

@article{length3,
    author = {Kunzinger, M. and S\"{a}mann, C.},
    title = {Lorentzian length spaces},
    journal = {Ann. Glob. Anal. Geom.},
    volume = 54,
    number = 3,
    year = 2018,
    pages = {399-447}
}

@article{area,
    author = {Chruściel, P. T. and Delay, E. and Galloway, G. J. and Howard, R.},
    title = {Regularity of horizons and the area theorem},
    journal = {Annales Henri Poincaré},
    volume = 2,
    year = 2001,
    pages = {109-178}
}

@article{c11-1,
    author = {Chruściel, P. T. and Grant, J. D. E.},
    title = {On Lorentzian causality with continuous metrics},
    journal = {Class. Quantum Gravity},
volume = {29},
number = {14},
    year = {2012}
}

@article{c11-3,
    author = {Kunzinger, M. and Steinbauer, R. and Stojković, M.},
    title = {The exponential map of a $C^{1,1}$-metric},
    journal = {Differ. Geom. Appl.},
   volume = {34},
    year = {2014},
pages = {14--24}
}

@article{c11-5,
    author = {Kunzinger, M. and Steinbauer, R. and Stojković, M.},
    title = {Hawking’s singularity theorem for $C^{1,1}$-metrics},
    journal = {Class. Quantum Gravity},
    volume = {32},
    number = {7},
    year = {2015}
}

@article{c11-6,
    author = {Kunzinger, M. and Steinbauer, R. and Vickers, J. A.},
    title = {The Penrose singularity theorem in regularity $C^{1,1}$},
    journal = {Class. Quantum Gravity},
    volume = {32},
    number = {15},
    year = {2015}
}

@article{minguzzicones,
    author = {Minguzzi, E.},
    title = {Causality theory for closed cone structures with applications},
    journal = {Rev. Math. Phys.},
    volume = {31},
    number = {5},
    year = {2019}
}

@article{Hau_2020,
    author = {Hau, L. A. and Cabrera Pacheco, A. J. and Solis, D. A.},
    title = {On the causal hierarchy of Lorentzian length spaces},
    journal = {Classical and Quantum Gravity},
    year = {2020}
}

@article{emmrich,
    author = {Emmrich, C. and Römer, H.},
    title = {Orbifolds as configuration spaces of systems with gauge symmetries},
    journal = {Commun. Math. Phys.},
    volume = {129},
    year = {1990},
    pages = {69--94}
}

@article{Length1,
    author = {Kunzinger, M. and S\"{a}mann, C.},
    title = {Lorentzian length spaces},
    journal = {Ann. Global Anal. Geom.},
    volume = {54},
    number = 3,
    year = 2018,
    pages = {399--447}
}

@article{conefields1,
    author = {Bernard, P. and Suhr, S.},
    title = {Lyapounov functions of closed cone fields from Conley theory to time functions},
    journal = {Comm. Math. Phys.},
    volume = 359,
    number = 2,
    year = {2018},
    pages = {467--498}
}

@article{Fathi,
    author = {Fathi, A. and Siconolfi, A.},
    title = {On smooth time functions},
    journal = {Mathematical Proceedings of the Cambridge Philosophical Society},
    volume = 152,
    number = 2,
    year = {2012},
    pages = {303--339}
}

@article{boyd,
    author = {Boyd, S. and Vandenberge, L.},
    title = {Convex optimization},
    journal = {Cambridge university press},
    year = {2009}
}

@article{cinfbaire,
    author = {Gravesen, J.},
    title = {Whitney $C^\infty$-topologies and the Baire property},
    journal = {Mathematica Scandinavica},
    volume = {52},
    number = {1},
    year = {1983},
    pages = {58--60}
}

@book{tti,
   title =     {Topology Through Inquiry},
   author =    {Starbird M. and Su, F. E.},
   publisher = {MAA Press},
   isbn =      {1470452766; 9781470452766},
   year =      {2019},
   series =    {AMS/MAA Textbooks 58}}

@article{mondino,
    author = {Cavalletti, F. and Mondino, A.},
    title = {A review of Lorentzian synthetic theory of timelike Ricci curvature bounds},
    journal = {Gen. Rel. and Gravitation},
    volume = 54,
    number = 11,
    year = 2022
}

@article{kron-pen,
    author = {E. H. Kronheimer, E.H and R. Penrose, R.},
    title = {On the structure of causal spaces},
    journal = {Mathematical 
Proceedings of the Cambridge Philosophical Society},
    volume = {63},
    pages = {481--505},
    year = {1967}
}

@book{H1967,
author = {H. Busemann},
location = {Warszawa},
publisher = {Instytut Matematyczny Polskiej Akademi Nauk},
title = {Timelike spaces},
url = {http://eudml.org/doc/268386},
year = {1967},
}

@article{javasanchez,
doi = {10.1088/0264-9381/25/16/168001},
url = {https://doi.org/10.1088/0264-9381/25/16/168001},
year = {2008},
month = {aug},
publisher = {},
volume = {25},
number = {16},
pages = {168001},
author = {Javaloyes, Miguel Angel and Sánchez, Miguel},
title = {A note on the existence of standard splittings for conformally stationary spacetimes},
journal = {Classical and Quantum Gravity}
}

@article{levin,
    author = {Levin, V. L.},
    title = {A continuous utility theorem for closed preorders on a $\sigma$-compact metrizable space},
    journal = {Soviet Math. Dokl},
    year = {1983}
}

@article{mingutility,
    author = {Minguzzi, E.},
    title = {Time Functions as Utilities},
    journal = {Communications in Mathematical Physics},
    year = {2010}
}

@article{minguzzicausality,
    author = {Minguzzi, E.},
    title = {Lorentzian causality theory},
    journal = {Living Reviews in Relativity},
    DOI = {10.1007/s41114-019-0019-x},
    number = {22,3},
    year = {2019}
    
}

@article{Zung_2024,
	title={Reeb flows transverse to foliations},
	volume={28},
	ISSN={1465-3060},
	url={http://dx.doi.org/10.2140/gt.2024.28.3661},
	DOI={10.2140/gt.2024.28.3661},
	number={8},
	journal={Geometry \& Topology},
	publisher={Mathematical Sciences Publishers},
	author={Zung, Jonathan},
	year={2024},
	month=Dec, pages={3661–3695} 
	}

@article{PMIHES_1981__54__5_0,
	author = {Sato, Atsushi and Tamura, Itiro},
	title = {On transverse foliations},
	journal = {Publications Math\'ematiques de l'IH\'ES},
	pages = {5--35},
	year = {1981},
	publisher = {Institut des Hautes \'Etudes Scientifiques},
	volume = {54},
	doi = {10.1007/BF02698690},
	mrnumber = {83k:57022},
	zbl = {0484.57016},
	language = {en},
	url = {https://www.numdam.org/articles/10.1007/BF02698690/}
}

@article{GoertschesToben+2018+1+40,
	url = {https://doi.org/10.1515/crelle-2015-0102},
	title = {Equivariant basic cohomology of Riemannian foliations},
	title = {},
	author = {Oliver Goertsches and Dirk Töben},
	pages = {1--40},
	volume = {2018},
	number = {745},
	journal = {Journal für die reine und angewandte Mathematik (Crelles Journal)},
	doi = {doi:10.1515/crelle-2015-0102},
	year = {2018},
	lastchecked = {2026-06-30}
}

@article{royo2005top,
	title={Top-Dimensional Group of the Basic Intersection Cohomology for Singular Riemannian Foliations},
	author={Royo Prieto, Jos{\'e} and Saralegi-Aranguren, Martintxo and Wolak, Robert},
	journal={Bulletin of the Polish Academy of Sciences. Mathematics},
	volume={53},
	number={4},
	pages={429--440},
	year={2005},
	publisher={Institute of Mathematics Polish Academy of Sciences}
}

@article{Habib_2024,
	title={New cohomological invariants of foliations},
	volume={28},
	ISSN={1945-0036},
	url={http://dx.doi.org/10.4310/AJM.250314015548},
	DOI={10.4310/ajm.250314015548},
	number={6},
	journal={Asian Journal of Mathematics},
	publisher={International Press of Boston},
	author={Habib, Georges and Richardson, Ken},
	year={2024},
	pages={757–790} }

@article{10.1093/qmath/haaa071,
	author = {Álvarez López, Jesús A and Jung, Seoung Dal},
	title = {Transversal Hard Lefschetz Theorem on Transversely Symplectic Foliations},
	journal = {The Quarterly Journal of Mathematics},
	volume = {72},
	number = {4},
	pages = {1235-1251},
	year = {2021},
	month = {12},
	issn = {0033-5606},
	doi = {10.1093/qmath/haaa071},
	url = {https://doi.org/10.1093/qmath/haaa071},
	eprint = {https://academic.oup.com/qjmath/article-pdf/72/4/1235/43803079/haaa071.pdf},
}

@article{10.1093/qmath/hax051,
	author = {Lin, Yi},
	title = {Hodge theory on transversely symplectic foliations},
	journal = {The Quarterly Journal of Mathematics},
	volume = {69},
	number = {2},
	pages = {585-609},
	year = {2018},
	month = {06},
	issn = {0033-5606},
	doi = {10.1093/qmath/hax051},
	url = {https://doi.org/10.1093/qmath/hax051},
	eprint = {https://academic.oup.com/qjmath/article-pdf/69/2/585/25015087/hax051.pdf},
}

@article{dolgonosova2018pseudo,
	title={Pseudo-Riemannian foliations and their graphs},
	author={Dolgonosova, A Yu and Zhukova, NI},
	journal={Lobachevskii Journal of Mathematics},
	volume={39},
	number={1},
	pages={54--64},
	year={2018},
	publisher={Springer}
}

@misc{caramello2025hadamardtheoremtransverselyaffine,
	title={A Hadamard theorem in transversely affine geometry with applications to affine orbifolds}, 
	author={Francisco C. Caramello Jr and Henrique A. Puel Martins and Ivan P. Costa e Silva},
	year={2025},
	eprint={2503.06344},
	archivePrefix={arXiv},
	primaryClass={math.DG},
	url={https://arxiv.org/abs/2503.06344}, 
}

\end{document}